\documentclass[12pt]{amsart}
\usepackage{amssymb,amsmath}
\usepackage{bm}
\usepackage{amsfonts}
\usepackage{epsf}
\usepackage{graphicx}
\usepackage[caption=false]{subfig}
\usepackage{tikz}
\usetikzlibrary{arrows,decorations.markings}
 \tikzstyle{vecArrow} = [thick, decoration={markings,mark=at position
   1 with {\arrow[semithick]{open triangle 60}}},
   double distance=2.5pt, shorten >= 5.5pt, shorten <= .5pt,
   preaction = {decorate}]
\usetikzlibrary{calc} 

\newtheorem{theorem}{Theorem}[section]
\newtheorem{lemma}[theorem]{Lemma}
\newtheorem{corollary}[theorem]{Corollary}

\newtheorem{remark}[theorem]{Remark}

\theoremstyle{definition}
\newtheorem*{assumption}{Assumption}
\numberwithin{equation}{section}

\newcommand{\bc}{\begin{center}}
\newcommand{\ec}{\end{center}}
\newcommand{\be}{\begin{eqnarray}}
\newcommand{\ee}{\end{eqnarray}}

\newcommand{\ben}{\begin{eqnarray*}}
\newcommand{\een}{\end{eqnarray*}}

\newcommand{\Om}{\Omega}

\newcommand{\eps}{\varepsilon}

\newcommand{\om}{\omega}
\newcommand{\pa}{\partial}

\newcommand{\na}{\nabla}

\def\x{\times}

\def\na{\nabla}
\def\dx{\,dx}

\def\ds{\,ds}

\def\cF{\mathcal{E}}
\def\cF{\mathcal{F}}

\def\cT{\mathcal{T}}
\def\R{\mathbb{R}}

\newcommand{\cM}{\ensuremath{\mathcal{M}} }
\newcommand{\NN}{\ensuremath{\mathcal{N}} }

\def\div{\operatorname{div}}

\DeclareMathOperator{\ddiv}{div}

\DeclareMathOperator{\Curl}{curl}

\DeclareMathOperator{\RM}{RM}

\newcommand{\step}[1]{\par\medskip\noindent{\em Step #1.}}
\newcommand{\case}[1]{\par\medskip\noindent{\em Case #1.}}
\newcommand{\ennorm}[1]{\lvert\!\lvert\!\lvert #1 \rvert\!\rvert\!\rvert}
\newcommand{\tri}{\mathcal{T}}
\newcommand{\SCone}{S^{\mathrm{C},1}_D}
\newcommand{\SCtwo}{S^{\mathrm{C},2}_D}
\newcommand{\Snc}{S^{\mathrm{NC},1}_D}

\title[Low-order nonconforming FEMs in 3D]
{Two low-order nonconforming finite element methods for the Stokes flow in 3D}
\author[J. Hu and M. Schedensack]{Jun Hu and Mira Schedensack}
\address{Jun Hu, LMAM and School of Mathematical Sciences,
 Peking University, Beijing 100871, P. R. China}
 \email{hujun@math.pku.edu.cn}
 \address{Mira Schedensack, Institut f\"ur Numerische und Angewandte Mathematik, 
 Universit\"at M\"unster, Einsteinstr. 62, D-48149 M\"unster, Germany}
 \email{mira.schedensack@uni-muenster.de}
\thanks{The first author was supported by  the NSFC Project 11625101. 
The second author 
       gratefully acknowledges support by the DFG
       Priority Program 1748 
       under the project ``Robust and Efficient Finite Element Discretizations 
for Higher-Order Gradient Formulations''
(SCHE1885/1-1).
Parts of this article were written 
while the second author enjoyed the kind hospitality of 
the Hausdorff Institute for Mathematics (Bonn, Germany). 
Both authors would like to thank the supports of the fifth and sixth 
Chinese-German workshops on Applied and Computational Mathematics held at the 
University of Augsburg, Germany, from September 21st to 25th, 2015, and the 
Tongji  University, China, from October 9th to 13rd, 2017.}
\date{\today}

\begin{document}

\begin{abstract}
In this paper, we propose two low order nonconforming finite element meth\-ods (FEMs)
for the three-dimensional Stokes flow that generalize the non-con\-for\-ming 
FEM of Kouhia and Stenberg (1995, Comput. Methods Appl. Mech. Engrg.). 
The finite element spaces proposed in this paper consist of two globally 
continuous components (one piecewise affine and one enriched component)
and one component that is continuous at the midpoints of interior faces.
We prove that the discrete Korn inequality and a discrete inf-sup condition
hold uniformly in the meshsize and also for a non-empty Neumann boundary.
Based on these two results, we show the  well-posedness
of the discrete problem.
Two counterexamples prove that there is no direct generalization of 
the Kouhia-Stenberg FEM to three space dimensions: The finite element space 
with one non-conforming and two conforming piecewise affine components does not satisfy 
a discrete inf-sup condition with piecewise constant pressure approximations,
while finite element functions with
two non-conforming and one conforming component do not satisfy a discrete 
Korn inequality.
\end{abstract}

\maketitle

\section{Introduction}

Given a polygonal, bounded Lipschitz domain $\Omega\subseteq\R^3$ with 
closed Dirichlet boundary $\Gamma_D$ and Neumann boundary 
$\Gamma_N=\partial \Omega\setminus\Gamma_D$ both with positive two-dimensional measure
and some right-hand side $g\in[L^2(\Omega)]^3$, the three dimensional 
Stokes problem seeks the velocity $u\in[H^1(\Omega)]^3$ and the pressure 
$p\in L^2(\Omega)$ with 
\begin{align} \label{e:StokesStrong}
  \left. 
  \begin{aligned}
    -2\mu\ddiv\eps(u) + \nabla p &= g\\
    \div u &=0
  \end{aligned}
  \right\}\; \text{in }\Omega,
  \quad
  u\vert_{\Gamma_D}=0,
  \quad
  (2\mu\varepsilon(u)-p I_{3\times 3})\vert_{\Gamma_N} \nu = 0.
\end{align}
Here and throughout this paper, $\mu$ is the viscosity.  The symmetric gradient of
a vector field reads $\eps(v):=1/2(\na v+\na v^T)$  for any $v\in [H^1(\Om)]^3$,
while $\nu$ denotes the outer unit normal.

Finite element methods (FEMs) for 
the two  dimensional Stokes problem have
been extensively studied in the literature, most of stable schemes are summarised
in the  book \cite{BoffiBrezziFortin2013}.  However only little
attention has been paid to the three dimensional problem. 
Here, we only mention the 
works~\cite{Stenberg1987,Boffi1997,Zhang2005,GuzmanNeilan2014,NeilanSap2016} 
for the three
dimensional Taylor-Hood elements. FEMs with discontinuous ansatz functions
for the pressure, and therefore, an improved mass-conservation are introduced 
in \cite{BernardiRaugel1985,BoffiCavalliniGardiniGastaldi2012a,BoffiCavalliniGardiniGastaldi2012b}.
If $\Gamma_N=\emptyset$, the Stokes equations can be reformulated in 
terms of the full gradient of $u$. In this case,
the non-conforming FEM of Crouzeix and Raviart \cite{CrouzeixRaviart1973} 
yields a stable approximation. Otherwise, it is not stable 
due to a missing Korn inequality
in two as well as in three dimensions. In 2D, the non-conforming FEM of Kouhia and 
Stenberg~\cite{KouhiaStenberg1995} circumvents this by choosing 
only one component non-conforming and the other one conforming. 
This non-conforming FEM is the lowest-order FEM for the Stokes problem 
with piecewise constant pressure approximation in 2D.
A generalization to higher polynomial degrees of that FEM can be found 
in \cite{Schedensack:2016}.

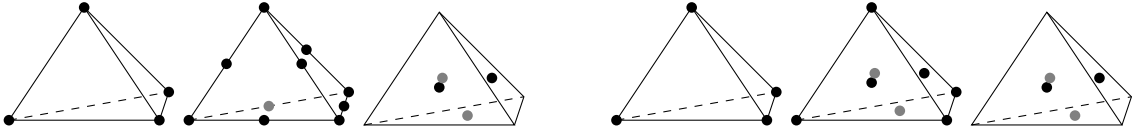
\begin{figure}
 \begin{center}
  \begin{tikzpicture}
    \draw (0,0)--(2,0)--(1,1.5)--cycle;
    \draw (2,0)--(2.125,0.375)--(1,1.5);
    \draw[dashed] (0,0)--(2.125,0.375);
    \fill (0,0) circle (2pt);
    \fill (2,0) circle (2pt);
    \fill (1,1.5) circle (2pt);
    \fill (2.125,0.375) circle (2pt);
  \end{tikzpicture}
  \begin{tikzpicture}
    \draw (0,0)--(2,0)--(1,1.5)--cycle;
    \draw (2,0)--(2.125,0.375)--(1,1.5);
    \draw[dashed] (0,0)--(2.125,0.375);
    \fill (0,0) circle (2pt);
    \fill (2,0) circle (2pt);
    \fill (1,1.5) circle (2pt);
    \fill (2.125,0.375) circle (2pt);
    \fill (1,0) circle (2pt);
    \fill (0.5,0.75) circle (2pt);
    \fill (1.5,0.75) circle (2pt);
    \fill (2.0625,0.1875) circle (2pt);
    \fill (1.5625,0.9375) circle (2pt);
    \fill[gray] (1.0625,0.1875) circle (2pt);
  \end{tikzpicture}
  \begin{tikzpicture}
    \draw (0,0)--(2,0)--(1,1.5)--cycle;
    \draw (2,0)--(2.125,0.375)--(1,1.5);
    \draw[dashed] (0,0)--(2.125,0.375);
    \fill (1,0.5) circle (2pt);
    \fill (1.7,0.625) circle (2pt);
    \fill[gray] (1.0417,0.625) circle (2pt);
    \fill[gray] (1.375,0.125) circle (2pt);
  \end{tikzpicture}
  \hspace{0.9cm}
  \begin{tikzpicture}
    \draw (0,0)--(2,0)--(1,1.5)--cycle;
    \draw (2,0)--(2.125,0.375)--(1,1.5);
    \draw[dashed] (0,0)--(2.125,0.375);
    \fill (0,0) circle (2pt);
    \fill (2,0) circle (2pt);
    \fill (1,1.5) circle (2pt);
    \fill (2.125,0.375) circle (2pt);
  \end{tikzpicture}
  \begin{tikzpicture}
    \draw (0,0)--(2,0)--(1,1.5)--cycle;
    \draw (2,0)--(2.125,0.375)--(1,1.5);
    \draw[dashed] (0,0)--(2.125,0.375);
    \fill (0,0) circle (2pt);
    \fill (2,0) circle (2pt);
    \fill (1,1.5) circle (2pt);
    \fill (2.125,0.375) circle (2pt);
    \fill (1,0.5) circle (2pt);
    \fill (1.7,0.625) circle (2pt);
    \fill[gray] (1.0417,0.625) circle (2pt);
    \fill[gray] (1.375,0.125) circle (2pt);
  \end{tikzpicture}
  \begin{tikzpicture}
    \draw (0,0)--(2,0)--(1,1.5)--cycle;
    \draw (2,0)--(2.125,0.375)--(1,1.5);
    \draw[dashed] (0,0)--(2.125,0.375);
    \fill (1,0.5) circle (2pt);
    \fill (1.7,0.625) circle (2pt);
    \fill[gray] (1.0417,0.625) circle (2pt);
    \fill[gray] (1.375,0.125) circle (2pt);
  \end{tikzpicture}
 \end{center}
 \caption{\label{f:dofs}Degrees of freedom of the velocity approximation 
  for the two new stable finite elements. The pressure is approximated 
  with piecewise constants.}
\end{figure}

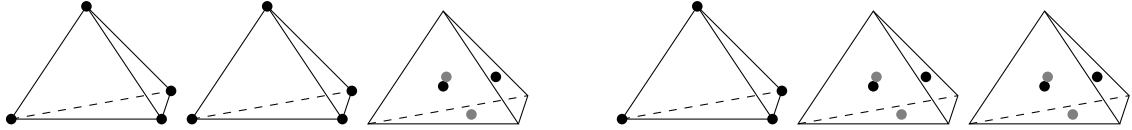
\begin{figure}
 \begin{center}
  \begin{tikzpicture}
    \draw (0,0)--(2,0)--(1,1.5)--cycle;
    \draw (2,0)--(2.125,0.375)--(1,1.5);
    \draw[dashed] (0,0)--(2.125,0.375);
    \fill (0,0) circle (2pt);
    \fill (2,0) circle (2pt);
    \fill (1,1.5) circle (2pt);
    \fill (2.125,0.375) circle (2pt);
  \end{tikzpicture}
  \begin{tikzpicture}
    \draw (0,0)--(2,0)--(1,1.5)--cycle;
    \draw (2,0)--(2.125,0.375)--(1,1.5);
    \draw[dashed] (0,0)--(2.125,0.375);
    \fill (0,0) circle (2pt);
    \fill (2,0) circle (2pt);
    \fill (1,1.5) circle (2pt);
    \fill (2.125,0.375) circle (2pt);
  \end{tikzpicture}
  \begin{tikzpicture}
    \draw (0,0)--(2,0)--(1,1.5)--cycle;
    \draw (2,0)--(2.125,0.375)--(1,1.5);
    \draw[dashed] (0,0)--(2.125,0.375);
    \fill (1,0.5) circle (2pt);
    \fill (1.7,0.625) circle (2pt);
    \fill[gray] (1.0417,0.625) circle (2pt);
    \fill[gray] (1.375,0.125) circle (2pt);
  \end{tikzpicture}
  \hspace{0.9cm}
  \begin{tikzpicture}
    \draw (0,0)--(2,0)--(1,1.5)--cycle;
    \draw (2,0)--(2.125,0.375)--(1,1.5);
    \draw[dashed] (0,0)--(2.125,0.375);
    \fill (0,0) circle (2pt);
    \fill (2,0) circle (2pt);
    \fill (1,1.5) circle (2pt);
    \fill (2.125,0.375) circle (2pt);
  \end{tikzpicture}
  \begin{tikzpicture}
    \draw (0,0)--(2,0)--(1,1.5)--cycle;
    \draw (2,0)--(2.125,0.375)--(1,1.5);
    \draw[dashed] (0,0)--(2.125,0.375);
    \fill (1,0.5) circle (2pt);
    \fill (1.7,0.625) circle (2pt);
    \fill[gray] (1.0417,0.625) circle (2pt);
    \fill[gray] (1.375,0.125) circle (2pt);
  \end{tikzpicture}
  \begin{tikzpicture}
    \draw (0,0)--(2,0)--(1,1.5)--cycle;
    \draw (2,0)--(2.125,0.375)--(1,1.5);
    \draw[dashed] (0,0)--(2.125,0.375);
    \fill (1,0.5) circle (2pt);
    \fill (1.7,0.625) circle (2pt);
    \fill[gray] (1.0417,0.625) circle (2pt);
    \fill[gray] (1.375,0.125) circle (2pt);
  \end{tikzpicture}
 \end{center}
 \caption{\label{f:KSnotstable}Degrees of freedom of the velocity for the 
  direct generalization
  of the 2D FEM of Kouhia and Stenberg \cite{KouhiaStenberg1995}. Those two 
  FEMs are not stable, as shown in Section~\ref{s:counterexamples}.}
\end{figure}

One key result of this paper consists in
two counterexamples in Section~\ref{s:counterexamples} below which 
imply that a generalization to three dimensions with 
two conforming and one non-conforming component is not inf-sup stable, 
while a generalization with one conforming and two non-conforming 
components does not satisfy a discrete Korn inequality, 
see Figure~\ref{f:KSnotstable} for a visualization of the degrees 
of freedom for those two FEMs.
To ensure both a discrete inf-sup stability and a discrete Korn 
inequality, we employ a discrete space 
consisting of one piecewise affine and globally continuous
and one non-conforming piecewise affine component. 
The third component can be approximated in the space of piecewise quadratic
and globally continuous functions as well as in the space of piecewise affine 
and globally continuous functions enriched with face bubble functions,
see Figure~\ref{f:dofs} for an illustration of the degrees of freedom.
The discrete inf-sup condition and the discrete Korn inequality imply 
the well posedness of the method. 
Furthermore, the recently established medius analysis technique 
\cite{Gudi2010,BadiaCodinaGudiGuzman2014,HuMaShi2014,CarstensenKoehlerPeterseimSchedensack2015,
CarstensenSchedensack2014}
together with the a~posteriori techniques of \cite{Carstensen05, CarstensenHu2007} 
proves a best-approximation result for the non-conforming FEM;
see Theorem~\ref{t:bestapprox} below.

The rest of the paper is organised as follows. In Section~\ref{s:FEM}, we present the
finite element method for \eqref{e:StokesStrong}.  The well-posedness of the discrete
problem will be proved in Section~\ref{s:stab}. 
Two counterexamples are given in Section~\ref{s:counterexamples}
that prove that discretizations with piecewise affine approximations 
for the velocity and piecewise constant approximations for the pressure 
are not stable.
Section~\ref{s:numerics} concludes the paper with numerical experiments.

Throughout this paper, standard notation on Lebesgue and Sobolev spaces is 
employed and $(\bullet,\bullet)_{L^2(\Omega)}$ denotes the $L^2$ scalar product over 
$\Omega$. Let $\|\bullet\|_{0,\omega}$ denote the $L^2$ norm over a set 
$\omega\subseteq\Omega$ (possibly two-dimensional) and $\|\bullet\|_0$
abbreviates $\|\bullet\|_{0,\Omega}$.  
The space $H_D^1(\Omega)$ consists of all 
$H^1$ functions that vanish on $\Gamma_D$ in the sense of traces.
Let $A\lesssim B$ abbreviate that there exists some mesh-size independent 
generic constant $0\leq C<\infty$ such that $A\leq C B$ and let $A\approx B$
abbreviate $A\lesssim B\lesssim A$.

\section{Finite element method}\label{s:FEM}

Suppose that the closure $\overline{\Omega}$ is covered
exactly by a regular and shape-regular triangulation $\cT$ of $\overline{\Omega}$ into
closed tetrahedra in $3D$ in the sense of Ciarlet \cite{BrennerScott08}, that is 
two distinct tetrahedra are either disjoint or share exactly one 
vertex, edge or face.
Let $\cF$ denote the set of all faces in $\cT$
with $\cF(\Om)$  the set of interior faces, $\cF(\Gamma_D)$ the set of faces on 
$\Gamma_D$, and $\cF(\Gamma_N)$ the set of faces on $\Gamma_N$.
Let $\mathcal{N}$ be the set of all vertices   with $\NN(\Om)$ the set of interior vertices, 
$\NN(\Gamma_D)$ the set of vertices  on $\Gamma_D$, and $\NN(\Gamma_N)$ the set 
of vertices on $\Gamma_N$.  The set of faces of the element $T$ is denoted  by $\cF(T)$.
By $h_T$ we denote the diameter of the element $T\in \cT$ and by $h_\tri$ the 
piecewise constant mesh-size function with $h_\tri\vert_T=h_T$ for all $T\in\tri$. 
We denote
by $\omega_T$ the union of (at most five) tetrahedra $T'\in \cT$  that  share a face
with $T$,  and by $\omega_F$ the union of (at most two) tetrahedra having in common the
face $F$.  Given any face $F\in\cF(\Om)$ with diameter  $h_F$ we assign one
fixed unit normal $\nu_F:=(\nu_1,\, \nu_2,\nu_3)$.
For $F$ on the boundary we choose $\nu_F=\nu$ the unit outward normal to
$\Omega$. Once $\nu_F$ has  been fixed on $F$, in relation to $\nu_F$ one
defines the elements $T_{-}\in \cT$ and $T_{+}\in \cT$,
with $F=T_{+}\cap T_{-}$ and $\om_F=T_{+}\cup T_{-}$, such that $\nu_F$ is 
the outward normal of $T_+$.
Given $F\in\cF(\Omega)$ and some $\R^d$-valued function $v$ defined
in $\Omega$, with $d=1\,,2\,,3$, we denote by
$[v]:=(v|_{T_+})|_F-(v|_{T_-})|_F$ the jump of $v$ across $F$ which will 
become the trace on boundary faces.

Let $P_0(T)$ denote the space of constant functions on $T$, 
$P_1(T)$ the space of affine functions and 
$P_2(T)$ the space of quadratic functions and 
let $\SCone$ and $\SCtwo$ denote the piecewise affine and piecewise 
quadratic conforming finite element spaces over $\cT$ which read
\begin{align*}
  \SCone&:=\left\{v\in H^1_D(\Omega)\left|\;\forall T\in\cT,\; v|_T\in P_1(T)
    \right\}\right.\,,\\
  \SCtwo&:=\left\{v\in H^1_D(\Omega)\left|\;\forall T\in\cT,\; v|_T\in P_2(T)
    \right\}\right.\, .
\end{align*}
The nonconforming  linear finite
element space $\Snc$  is defined as
\begin{equation*}
  \Snc:=\left\{v\in L^2(\Omega)\left|
    \begin{array}{l}
      \forall T\in\cT,\; v|_T\in P_1(T),\;
      \forall F\in \cF(\Om),\; \int_{F}[v]_F\ds=0,\\
      \text{and }\forall F\in\cF\text{ with }F\subseteq\Gamma_D,\;
      \int_F v\ds=0
    \end{array}
    \right\}\right. \,.
\end{equation*}
Define also the space of face bubbles by 
\begin{align*}
 \mathcal{B}_\mathcal{F}
   :=\operatorname{span}\{\varphi_F\vert 
       F\in\cF(\Omega)\text{ or }F\in\cF\text{ with }F\subseteq\Gamma_N\}
\end{align*}
with the face bubbles $\varphi_F\in H^1_D(\Omega)$ defined by 
\begin{align*}
 \varphi_F:=60 \lambda_a\lambda_b\lambda_c
  \quad\text{ for }\quad F=\operatorname{conv}\{a,b,c\}
\end{align*}
and with barycentric coordinates $\lambda_a,\lambda_b,\lambda_c$.
We consider two finite element spaces for the velocity. 
The first one is the space which contains second order polynomials in 
the second component and it is defined by
$$
V_{2,D}:=\SCone\x \SCtwo\x \Snc\,.
$$
As a second finite element space for the velocity we consider the 
enrichment of the second component by face bubbles, i.e.,
$$ 
 V_{\cF,D}:=\SCone\x 
    (\SCone+\mathcal{B}_\cF)\x \Snc\,.
$$
Since $V_{2, D}$ and $V_{\cF,D}$ are nonconforming spaces, 
the differential operators $\na$, $\eps$
and $\div$ are defined elementwise, written as, $\na_h$, $\eps_h$ and $\div_h$,
respectively. We equip the space $V_{2, D}$ and $V_{\cF,D}$ with the broken norm
$$
\|v\|_{1,h}^2:=\|v\|_{0}^2+\|\na_h v\|_{0}^2
\quad \text{ for all } v\in V_{2, D}\oplus V_{\cF,D}\,.
$$
For both choices of finite element spaces for the velocity,
the pressure will be sought in the space 
\begin{align*}
 Q_h:=\{q\in L^2(\Omega)\vert \forall T\in\tri,\; q\vert_T\in P_0(T)\}
\end{align*}
consisting of piecewise constant functions.
Let $V_{h,D}$ be $V_{2,D}$ or $V_{\cF,D}$.
The finite element method then reads: Find $u_h\in V_{h, D}$ and $p_h\in Q_h$
with

\begin{equation}\label{StokesDisc}
\begin{aligned}
a_h(u_h, v_h)+b_h(v_h,p_h)& =(g,v_h)_{L^2(\Omega)}
 &&\quad \text{ for all } v_h\in V_{h, D}\,,\\
b_h(u_h,q_h)  &=0
 &&\quad  \text{ for all } q_h\in Q_h\,,
\end{aligned}
\end{equation}
where the two discrete bilinear forms read
\begin{align*}
a_h(u_h,v_h)&:=2\mu\int_\Omega\eps_h(u_h):\eps_h(v_h)\,dx,\\
b_h(v_h,p_h)&:=-\int_\Omega\ p_h \div_h v_h\,dx.
\end{align*}

The next section proves a discrete inf-sup condition and a discrete Korn
inequality. Those two ingredients then imply the existence of a unique 
solution from Corollary~\ref{c:uniqexdisc} and the best-approximation 
error estimate of Theorem~\ref{t:bestapprox} below. This leads to the 
convergence against the solution of the (weak form of)~\eqref{e:StokesStrong},
namely the solution $(u,p)\in [H^1_D(\Omega)]^3 \times L^2(\Omega)$ with 
\begin{equation}\label{e:StokesWeak}
\begin{aligned}
 2\mu(\eps(u),\eps(v))_{L^2(\Omega)} -(p,\div v)_{L^2(\Omega)} 
    &= (g,v)_{L^2(\Omega)}
  &&\quad\text{ for all }v\in [H^1_D(\Omega)]^3,\\
 (q,\div u)_{L^2(\Omega)} &= 0 
 &&\quad  \text{ for all } q\in L^2(\Omega)\, .
\end{aligned}
\end{equation}

\section{The stability  analysis}\label{s:stab}

In this section, we  prove the well-posedness
of the discrete problem and a best-approximation result, 
which follow from the discrete Korn inequality and
the inf-sup condition from Theorems~\ref{t:korn} and \ref{t:infsup} below.


The discrete Korn inequality relies on the following assumption.
\begin{assumption}[(H1)]
Any face $F\in\cF$ that lies on the Dirichlet boundary, $F\subseteq \Gamma_D$,
and that is horizontal in the sense that $\lvert\nu_F(3)\rvert=1$,
satisfies one of the following conditions:
\begin{enumerate}
 \item[(a)] There exists a vertex $z\in \NN$ and a face $F'\in\cF\setminus\{F\}$ such that 
       $z\in F\cap F'$, $F'\subseteq\Gamma_D$ and $F'$ is not horizontal in the 
       sense that $\lvert \nu_{F'}(3)\rvert < 1$.
 \item[(b)] There exist a vertex $z\in\NN$ and two faces $F',F''\in\cF$, $F'\neq F''$,
       such that $F'\subseteq\Gamma_D$, $F''\subseteq\Gamma_D$ and  
       $\{z\}=F\cap F'\cap F''$.
\end{enumerate}
\end{assumption}

Note that in condition (b) all of the faces $F$, $F'$ and $F''$ might be 
horizontal.

\begin{remark}
 The Assumption~(H1) basically excludes that there are horizontal 
 faces on the Dirichlet boundary which are surrounded by the 
 Neumann boundary. If Assumption~(H1) is not satisfied, the triangulation 
 can be refined with, e.g., a bisection algorithm such that vertices that 
 satisfy condition (b) are created. Note that assumption~(H1) is conserved 
 by a red, green or bisection refinement.
 
 The assumption~(H1) excludes the situation depicted in 
 Figure~\ref{f:counterexampleH1}, where an infinitesimal rigid body motion is not 
 excluded by the Dirichlet boundary condition, due to the non-conformity 
 in the ansatz space.
\end{remark}

\begin{remark}
A permutation of the conforming, non-conforming, and enriched finite element 
space in the definition of $V_{2,D}$ and $V_{\mathcal{F},D}$ is possible as 
well. The condition on horizontal faces in assumption (H1) has then be 
replaced by the corresponding condition on vertical faces with 
$\lvert\nu_F(1)\rvert=1$ or $\lvert\nu_F(2)\rvert=1$, corresponding to
the chosen non-conforming component. This might be beneficial in some 
situations.
\end{remark}

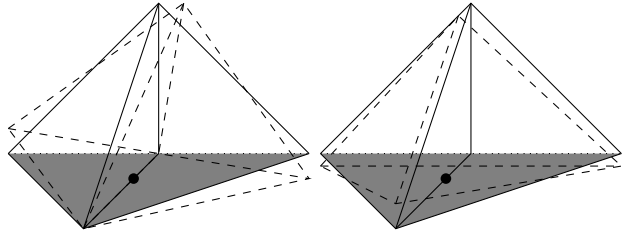
\begin{figure}
 \begin{tikzpicture}
 \fill[gray] (-2,0)--(-1,-1)--(2,0);
  \draw[dotted] (-2,0)--(2,0);
  \draw (-2,0)--(-1,-1)--(2,0);
  \draw (-2,0)--(0,2)--(2,0);
  \draw (0,2)--(-1,-1);
  \draw (0,2)--(0,0)--(-1,-1);
  \fill (-0.333,-0.333) circle (2pt);
  \draw[dashed] (-2,0.333)--(2,-0.333)--(-1,-1)--cycle;
  \draw[dashed] (-2,0.333)--(0.333,2)--(2,-0.333);
  \draw[dashed] (0.333,2)--(-1,-1);
  \draw[dashed] (0.333,2)--(0,0)--(-1,-1);
 \end{tikzpicture}
 \begin{tikzpicture}
 \fill[gray] (-2,0)--(-1,-1)--(2,0);
  \draw[dotted] (-2,0)--(2,0);
  \draw (-2,0)--(-1,-1)--(2,0);
  \draw (-2,0)--(0,2)--(2,0);
  \draw (0,2)--(-1,-1);
  \draw (0,2)--(0,0)--(-1,-1);
  \fill (-0.333,-0.333) circle (2pt);
  \draw[dashed] (-2,-0.1665)--(2,-0.1665)--(-1,-0.667)--cycle;  
  \draw[dashed] (-2,-0.1665)--(-0.1665,1.8335)--(2,-0.1665);
  \draw[dashed] (-0.1665,1.8335)--(-1,-0.667);
 \end{tikzpicture}
 \caption{\label{f:counterexampleH1}Two infinitesimal rigid body motions 
 that satisfy the Dirichlet boundary conditions at the gray face in 
 the non-conforming sense.}
\end{figure}

We furthermore assume the following assumption (H2).
\begin{assumption}[(H2)]
 There exists no interior face $F\in\cF(\Omega)$, whose three vertices 
 lie on the boundary $\partial\Omega$.
 Furthermore, the triangulation $\tri$ consists of more than one simplex.
\end{assumption}

\begin{remark}
Similar assumptions as~(H1) and~(H2) are necessary for the two dimensional 
situation. The assumption (H1) is hidden in \cite{KouhiaStenberg1995} 
in the assumption that the mesh-size 
has to be small enough. 
See also~\cite{CarstensenSchedensack2014} for a discussion about necessary 
 conditions on the triangulation.
\end{remark}

We define the set
\begin{align*}
  \NN(H1):=\left\{z\in\NN(\Gamma_D)\left\vert 
     \begin{array}{l}
       \exists F\in\cF \text{ such that } 
       F\subseteq\Gamma_D, \lvert\nu_F(3)\rvert =1 \\
       \text{and }z\text{ satisfies condition (a) or (b) for that }F
     \end{array}
     \right\}\right..
\end{align*}
Furthermore, given any vertex $z\in\NN(\Omega)$, we define 
\begin{align*}
  \cF_z:=\left\{F\in\cF\left\vert 
      \begin{array}{l}
      z\in F\text{ or }
      F\subseteq \Gamma_D\text{ with }\lvert\nu_F(3)\rvert < 1 \\
      \text{and }\exists T\in\tri \text{ with }
      z\in T\text{ and }F\subseteq T
      \end{array}
      \right\}\right.
\end{align*}
and for $z\in \NN(H1)$ we let 
\begin{align*}
   \cF_z:=\{F\in\cF\mid z\in F\}.
\end{align*}
denote the set of faces that share $z$.

\begin{remark}\label{r:overlap}
The assumptions (H1) and (H2) guarantee that any face $F\in \cF(\Omega)\cup\cF(\Gamma_D)$
is contained in a set $\cF_z$ for some node $z\in\NN(\Om)\cup\NN(H1)$.
\end{remark}

To establish the discrete Korn inequality,  we need the following key result.
Let $\omega_z$ denote the patch of $z$, i.e., 
\begin{align*}
 \omega_z:=\operatorname{int}\big(\bigcup \{T\in\tri\vert z\in T\}\big).
\end{align*}

\begin{lemma}\label{Lemmaefficiency}
Let (H1) and (H2) be satisfied.
Let $V_{h,D}$ be the finite element space $V_{2,D}$ or $V_{\cF,D}$ and
let $v_h\in V_{h,D}$.
For any vertex 
$z\in \NN(\Om)\cup \NN(H1)$, it holds that
\begin{equation}\label{efficiency}
\sum\limits_{F\in \cF_z}h_F^{-1}\|[v_h]_F\|_{0, F}^2
\leq C \inf\limits_{v\in[ H^1_D(\om_z)]^3}\|\eps_h(v-v_h)\|_{0, \om_z}^2\,,
\end{equation}
where
$$
H^1_D(\om_z):=\{v\in H^1(\om_z)\mid v=0 \text{ on } \pa\om_z\cap \Gamma_D\}\,.
$$
The constant $C$ may depend on the 
angles of the simplices and on the configuration of the simplices 
in $\omega_z$, but it is independent of the mesh-size.
\end{lemma}

\begin{proof}
In fact, both sides of \eqref{efficiency} define  seminorms for the 
restriction of  $v_h$ to $\om_z$.
 Suppose
$$
\inf\limits_{v\in [H^1_D(\om_z)]^3}\|\eps_h(v-v_h)\|_{0, \om_z}=0\,.
$$
This implies
\begin{align*}
v_h=v+\RM_z\,
\end{align*}
for some $v\in [H^1_D(\om_z)]^3$ and some piecewise rigid body motion $\RM_z$ which
 is of the form
$$
\RM_z|_T(x):=\begin{pmatrix}
        a_T-e_Tx_2-d_Tx_3\\
        b_T+e_Tx_1-f_Tx_3\\
        c_T+d_Tx_1+f_Tx_2\\
       \end{pmatrix}\quad \text{ for any } T\in \tri\text{ with }z\in T 
$$
for parameters $a_T\,,b_T\,,c_T\,,d_T\,,e_T\,,f_T$.
Assumption~(H2) guarantees that $\omega_z$ consists of more than one simplex.
Therefore,
consider a face $F\in\cF$ such that $F= T_1\cap T_2$ for some $T_k\in\tri$,
$T_k\subseteq\omega_z$ for $k=1,2$ and $F$ is not horizontal, i.e., 
$\lvert\nu_F(3)\rvert<1$. Then there exist parameters 
$\alpha,\beta,\gamma\in\R$ such that w.l.o.g.
\begin{align*}
  F\subseteq \{x\in\R^3\mid x_1+\alpha x_2 + \beta x_3=\gamma\}.
\end{align*}
Since the first two components of $v_h$ are continuous across internal faces, 
it follows 
\begin{align*}
  (a_{T_1}-a_{T_2}) - (e_{T_1}-e_{T_2}) x_2 - (d_{T_1}-d_{T_2})x_3&=0,\\
  (b_{T_1}-b_{T_2})+\gamma(e_{T_1}-e_{T_2}) - \alpha (e_{T_1}-e_{T_2}) x_2 
   - (f_{T_1}-f_{T_2}+\beta(e_{T_1}-e_{T_2}) )x_3&=0.
\end{align*}
Since this holds for all $x_2,x_3\in\R$, this leads to 
$$
a_{T_1}=a_{T_2},\;\;b_{T_1}=b_{T_2},\;\; d_{T_1}=d_{T_2},\; \;
e_{T_1}=e_{T_2},\;\;  f_{T_1}=f_{T_2}.
$$
Note that the integral mean of the last component of
$v_h$ is continuous across $F$. This leads to
$$
c_{T_1}=c_{T_2}.
$$
We conclude that $\RM_z$ is continuous across any face $F$ that is not 
horizontal. The same arguments prove that $\RM_z\vert_T$ vanishes, 
if $T$ contains a face $F\subseteq\Gamma_D$ that is not horizontal.

To conclude that $\RM_z$ is continuous on the whole patch $\omega_z$,
we let $\cF_{z,H}$ denote the set of faces in $\cF_z$ that are 
horizontal and consider the following cases.

\case{1: $\cF_{z,H}=\emptyset$}
In this case, $\RM_z$ is clearly continuous on $\omega_z$.

\case{2: $z\in\NN(\Om)$ and $\cF_{z,H}\neq\emptyset$} 
Let $D_\delta(z)$ denote the disc with radius $\delta$ 
and center $z$, i.e.,
\begin{align*}
 D_\delta(z):=\{x\in \omega_z\mid (x-z)\cdot (0,0,1)=0 
   \text{ and }\lvert x-z\rvert<\delta\}.
\end{align*}
We first consider the case that the faces in $\cF_{z,H}$ do not cover 
a whole disc, i.e., $D_\delta(z)\not\subseteq \bigcup \mathcal{F}_{z,H}$
for all $\delta>0$. Then $\omega_z\setminus\bigcup\cF_{z,H}$ is still connected and 
it follows that $\RM_z$ is continuous.

If the faces in $\cF_{z,H}$ contain a whole disc centred at $z$, i.e., 
there exists some $\delta>0$ such that 
$D_\delta(z)\subseteq \bigcup\mathcal{F}_{z,H}$, then  
the set $\om_z$ is  divided into two parts by the faces 
of $\mathcal{F}_{z,H}$.  
Let $\mathcal{F}_{z,H}^i\subseteq\mathcal{F}_{z,H}$ denote the set of 
those separating faces in $\mathcal{F}_{z,H}$ that are all faces that 
are not on $\Gamma_D$.
In each part, $\RM_z$ 
restricted to one of these parts is a global rigid body motion.  
The set $\mathcal{F}_{z,H}^i$
contains at least three faces, because $z$ is an interior vertex.
Since the jump across $\bigcup \mathcal{F}_{z,H}^i$
of the third component of $\RM_z$ is an affine function and vanishes at least 
at three different points that are not collinear, we have that $\RM_z$ is 
continuous.

If there exists some face $F\in\cF_z$ with $F\subseteq\Gamma_D$, then this 
face is not horizontal by the definition of $\cF_z$ for interior nodes.
Therefore, $\RM_z$ vanishes.

\case{3: $z\in\NN(H1)$ with $z$ from (a) from Assumption (H1)}
In this case, $\cF_z$ contains a face $F'\in\cF_z$ with $F'\subseteq\Gamma_D$ 
that is not horizontal. Therefore, $\RM_z$ vanishes.

\case{4: $z\in\NN(H1)$ with $z$ from (b) from Assumption (H1)}
In this case, there exist at least three faces $F,F',F''\in\cF_z$, 
which lie on the Dirichlet boundary and are horizontal.
Since the jump across $\bigcup \mathcal{F}_{z,H}$
of the third component of $\RM_z$ is an affine function and vanishes at least 
at the midpoints of these faces, we have that $\RM_z$ vanishes.

\medskip 
In all of the above cases, $\RM_z$ is continuous on $\omega_z$ and vanishes if 
$\cF_z$ contains Dirichlet boundary faces. If $F\in\cF_z$ and 
$F\not\subseteq \Gamma_D$, then $\operatorname{int}(F)\subseteq \omega_z$ for 
the relative interior $\operatorname{int}(F)$ of $F$. Therefore, 
$$
\sum\limits_{F\in \cF_z}h_F^{-1}\|[v_h]_F\|_{0, F}^2=0\,.
$$
In other words, the left-hand side 
of~\eqref{efficiency} vanishes. 
Hence, the two seminorms of the left and right side of~\eqref{efficiency}
satisfy 
\begin{align*}
 \sum\limits_{F\in \cF_z}h_F^{-1}\|[v_h]_F\|_{0, F}^2
\leq C \inf\limits_{v\in[ H^1_D(\om_z)]^3}\|\eps_h(v-v_h)\|_{0, \om_z}^2.
\end{align*}
A scaling argument shows that $C$ is independent of the mesh-size.
\end{proof}

\begin{remark}
 Note that the proof of Lemma~\ref{Lemmaefficiency} does only have to  
 control piecewise rigid body motions and, therefore, the proof 
 (and, hence, also Theorem~\ref{t:korn} below) holds true 
 for any choice of finite element space for the first and second component
 as long as they are conforming.
\end{remark}

With this lemma, we are in the position to prove the discrete Korn inequality.

\begin{theorem}\label{t:korn}
Assume that Assumptions (H1) and (H2) hold and that $h_F\lesssim 1$ for all 
$F\in\mathcal{F}(\Gamma_D)$. 
Let $V_{h,D}$ be $V_{2,D}$ or $V_{\cF,D}$.
There exists a positive constant $\beta$ independent of the
meshsize  such that
$$
 \beta \|v_h\|_{1,h}\leq \|\varepsilon_h(v_h)\|_0
 \qquad\text{ for all }  v_h\in V_{h,D}\,.
$$
\end{theorem}

\begin{proof}  
The discrete Poincare inequality from \cite[(1.5)]{Brenner2003} implies 
\begin{align*}
 \|v_h\|_{0}^2
   \lesssim \|\na_h v_h\|_{0}^2
     + \left\lvert\int_{\Gamma_D} v_h\,ds\right\rvert^2
   \lesssim \|\na_h v_h\|_{0}^2 
     +\sum_{F\in\cF(\Gamma_D)} h_F^{-1} \|[v_h]_F\|_{0,F}^2
\end{align*}
provided $h_F\lesssim 1$ for all $F\in\cF(\Gamma_D)$.
The discrete Korn inequality from~\cite[(1.19)]{Brenner2004} then leads to 
\begin{align*}
 \|v_h\|_{0}^2 + \|\na_h v_h \|_{0}^2
 &\leq C \left(\|\eps_h(v_h)\|_0^2 + \|v_h\|_{L^2(\Gamma_D)}^2 
   + \sum_{F\in\cF(\Omega)\cup\cF(\Gamma_D)} h_F^{-1} \|[v_h]_F\|_{L^2(F)}^2 \right)\\
 &\lesssim C \left(\|\eps_h(v_h)\|_0^2
   + \sum_{F\in\cF(\Omega)\cup\cF(\Gamma_D)} h_F^{-1} \|[v_h]_F\|_{L^2(F)}^2 \right).
\end{align*}
Lemma~\ref{Lemmaefficiency} and Remark~\ref{r:overlap} then yield the assertion.
\end{proof}

For the proof of the inf-sup condition, define 
for any interior vertex $z$ the associated
macroelement by
$$
\mathcal{M}=\mathcal{M}(z)=\{T\in\tri\mid z\in T\}
$$
and let 
\begin{align*}
 \Omega_\mathcal{M} = \operatorname{int}\big( \bigcup \mathcal{M}\big).
\end{align*}
Furthermore define the bilinear form $\mathcal{B}$ for all 
$v_h\in V_{2,D}\cup V_{\cF,D}$ and $q_h\in Q_h$ by
\begin{align*}
 \mathcal{B}(u_h,p_h;v_h,q_h)
   := a_h(u_h,v_h) + b_h(v_h,p_h) - b_h(u_h,q_h) 
\end{align*}
and the norm
\begin{align*}
 \ennorm{(v_h,q_h)}_h^2:=\|\nabla_h v_h\|_{L^2(\Omega)}^2 
   + \|q_h\|_{L^2(\Omega)}^2 .
\end{align*}

\begin{theorem}\label{t:infsup} 
Let $V_{h,D}$ be $V_{2,D}$ or $V_{\cF,D}$.
If Assumptions~(H1) and (H2) are satisfied, then 
there exists  a positive  constant $\alpha$ independent of the mesh-size, 
such that
\begin{align*}
 \sup_{(v_h,q_h)\in (V_{h,D}\times Q_h)\setminus\{0\}}
    \frac{\mathcal{B}(u_h,p_h;v_h,q_h)}{\ennorm{(v_h,q_h)}_h}
    \geq \alpha \ennorm{(u_h,p_h)}_h
  \qquad\text{for all }(u_h,p_h)\in V_{h,D}\times Q_h.
\end{align*}
\end{theorem}

\begin{proof}
The proof is divided into two steps.
\step{1}
We use the macroelement trick from \cite{KouhiaStenberg1995}.
To this end, let $z\in\NN(\Omega)$ be an interior node with macroelement 
$\mathcal{M}$.
Define 
\begin{align*}
 \widetilde{V}_{0,\mathcal{M}}&:=
   \left\{v\in [L^2(\Omega_\mathcal{M})]^3\left|
     \begin{array}{l}
       v=(v_1,v_2,v_3)\text{ with }v_1\in H^1_0(\Omega_\mathcal{M}),\
        v_2\in H^1_0(\Omega_\mathcal{M}),\\
        \forall T\in \mathcal{M}\ \forall i=1,3:\; v_i|_T\in P_1(T),\\
         \int_F[v_3]\ds=0 \text{ for any interior face $F$ of } \mathcal{M},\\
         \int_Fv_3\ds=0 \text{ for any face } F\subseteq \pa \Omega_\mathcal{M}
     \end{array}
    \right\}\right. \,,\\
  Q_{\mathcal{M}}&:=\{q\in L^2(\Omega_\mathcal{M})\mid  
  \forall T\in\mathcal{M},\; q|_T\text{ is constant}\}\,,\\
N_\mathcal{M}&:=\left\{q\in Q_\mathcal{M} \mid \forall v\in V_{0,\mathcal{M}}:
  \; \int_{\Omega}q\div_h v \dx=0\right\}\,.
\end{align*}
In the case that $V_{h,D}$ equals $V_{2,D}$, let
\begin{align*}
V_{0,\mathcal{M}}&:=
   \left\{v\in \widetilde{V}_{0,\mathcal{M}} \mid
        \forall T\in \mathcal{M}\; v_2\vert_T\in P_2(T)
    \right\} \,,
\end{align*}
while in the case that $V_{h,D}$ equals $V_{\cF,D}$, we set 
\begin{align*}
 V_{0,\mathcal{M}}&:=
   \left\{v\in \widetilde{V}_{0,\mathcal{M}} \mid
        \forall T\in \mathcal{M}\; 
        v_2\vert_T\in P_1(T)+\operatorname{span}\{\varphi_F\mid 
             F\in\cF(\Omega_\mathcal{M})\}
    \right\} \,.
\end{align*}
Let $q\in N_\mathcal{M}$.
Define for any $F$ with $\nu_F(3)\not=0$ a function 
$v\in V_{0,\mathcal{M}}$ by $v_1=0\,,v_2=0$, $\int_{F}v_{3}\ds=1$,  and
$\int_{F^\prime}v_{3}\ds=0$ for any face $F^\prime\neq F$. 
Then an integration by parts implies
\begin{align*}
0&=\int_{\omega_z} q \div_h v\,dx
 = \int_{T_{+}}\frac{\pa v_3}{\pa x_3}q\vert_{T_{+}}\dx
+\int_{T_{-}}\frac{\pa v_3}{\pa x_3}q\vert_{T_{-}}\dx
=\nu_{3}(q\vert_{T_{+}}-q\vert_{T_{-}})\,.
\end{align*}
Since $\nu_F(3)\neq 0$, this implies
$$
q\vert_{T_{+}}=q\vert_{T_{-}}\,.
$$
It follows that $q$ can only jump across vertical faces, i.e., if $\nu_F(3)=0$;
see Figure~\ref{f:illuproofinfsup1} for a possible configuration of vertical 
hyperplanes where $q$ can jump.

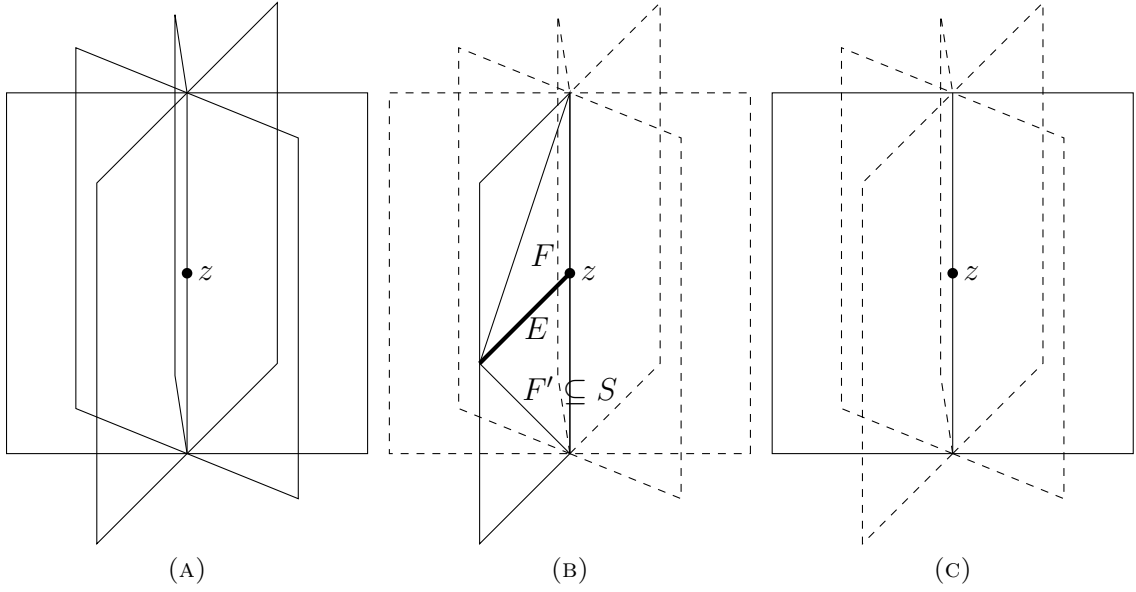
\begin{figure}
  \subfloat[\label{f:illuproofinfsup1}]{
  \begin{tikzpicture}[x=1.2cm,y=1.2cm]
    \draw (0,-2)--(0,2);
    \draw (0,-2)--(1,-1)--(1,3)--(0,2);
    \draw (2,-2)--(2,2)--(-2,2)--(-2,-2)--cycle;
    \draw (0,-2)--(1.2321,-2.5)--(1.2321,1.5)--(0,2);
    \draw (0,-2)--(-1,-3)--(-1,1)--(0,2);
    \draw (0,-2)--(-0.13397,-1.134)--(-0.13397,2.866)--(0,2);
    \draw (0,-2)--(-1.2321,-1.5)--(-1.2321,2.5)--(0,2);
    \fill (0,0) circle (2pt);
    \node (z) at (0,0)[right]{$z$};
  \end{tikzpicture}
  }
  \subfloat[\label{f:illuproofinfsup2}]{
  \begin{tikzpicture}[x=1.2cm,y=1.2cm]
    \draw (0,-2)--(0,2);
    \draw[dashed] (0,-2)--(1,-1)--(1,3)--(0,2);
    \draw[dashed] (2,-2)--(2,2)--(-2,2)--(-2,-2)--cycle;
    \draw[dashed] (0,-2)--(1.2321,-2.5)--(1.2321,1.5)--(0,2);
    \draw (0,-2)--(-1,-3)--(-1,1)--(0,2);
    \draw[dashed] (0,-2)--(-0.13397,-1.134)--(-0.13397,2.866)--(0,2);
    \draw[dashed] (0,-2)--(-1.2321,-1.5)--(-1.2321,2.5)--(0,2);
    \fill (0,0) circle (2pt);
    \node (z) at (0,0)[right]{$z$};
    \draw[ultra thick] (0,0)--(-1,-1);
    \node (midE) at (-0.37,-0.37)[below]{$E$};
    \draw (-1,-1)--(0,2)--(0,0);
    \node () at (-0.3,0.2){$F$};
    \draw (-1,-1)--(0,-2)--(0,0);
    \node () at (0,-1.3){$F'\subseteq S$};
  \end{tikzpicture}
  }
  \subfloat[\label{f:illuproofinfsup3}]{
  \begin{tikzpicture}[x=1.2cm,y=1.2cm]
    \draw (0,-2)--(0,2);
    \draw[dashed] (0,-2)--(1,-1)--(1,3)--(0,2);
    \draw (2,-2)--(2,2)--(-2,2)--(-2,-2)--cycle;
    \draw[dashed] (0,-2)--(1.2321,-2.5)--(1.2321,1.5)--(0,2);
    \draw[dashed] (0,-2)--(-1,-3)--(-1,1)--(0,2);
    \draw[dashed] (0,-2)--(-0.13397,-1.134)--(-0.13397,2.866)--(0,2);
    \draw[dashed] (0,-2)--(-1.2321,-1.5)--(-1.2321,2.5)--(0,2);
    \fill (0,0) circle (2pt);
    \node (z) at (0,0)[right]{$z$};
  \end{tikzpicture}
  }
  \caption{\label{f:illuproofinfsup}Illustration of vertical hyperplanes 
  in the proof of the inf-sup condition.}
\end{figure}

Let now $F$ be a vertical face.
We now have to treat the two different possible choices of ansatz spaces 
separately.

\textit{Case 1: $V_{h,D}=V_{2,D}$.}
Let $E\subseteq F$ be an edge of 
$\mathcal{M}$ that satisfies the following conditions.
\begin{itemize}
 \item $\nu_F\vert_E(2)\neq 0$,
 \item $E$ is an interior edge of $\Omega_\mathcal{M}$
    in the sense that $\mathrm{int}(E)\subseteq \Omega_\mathcal{M}$ for the 
    relative interior $\mathrm{int}(E)$ of $E$ (i.e., $E$ without its 
    endpoints),
 \item $E$ is not vertical, i.e., $\lvert x-z\rvert \neq \lvert (x-z)(3)\rvert$
    for all $x\in E\setminus\{z\}$.
\end{itemize}
See Figure~\ref{f:illuproofinfsup2} for an illustration of a possible 
configuration.
Define a function $v\in V_{0,\mathcal{M}}$ by $v=(0,v_2,0)$ with 
\begin{equation}\label{e:defv2proofinfsup}
\begin{aligned}
  v_2(\mathrm{mid}(E))&=1,\\  
 v_2(\mathrm{mid}(E'))&=0 \text{ for all edges }E'\text{ of }
\mathcal{M}\text{ with }E\neq E'\\
 \text{ and }
v_2(\tilde z)&=0\text{ for all nodes }\tilde z
\text{ of }\mathcal{M}. 
\end{aligned}
\end{equation}
Define 
\begin{align*}
  S:=\bigcup \{F'\mid E\subseteq F'\text{ and }F' \text{ is an interior face of }
\mathcal{M} \text{ and }\nu_{F'}(3)=0\}. 
\end{align*}
Since $q\in N_\mathcal{M}$ can only jump across vertical faces, an integration 
by parts proves 
\begin{align*}
  0=\int_{\Omega_\mathcal{M}} q\ddiv_h v\,dx 
   = \int_{\Omega_\mathcal{M}} q\frac{\partial v_2}{\partial x_2}\,dx 
   = \int_S [q]_S v_2 \nu_S(2)\,ds
   =  \nu_{S}(2)\, [q]_S\,  \int_S v_2\,ds\,.
\end{align*}
Since $\int_S v_2\,ds=\mathrm{area}(S)/12\neq 0$, this implies that 
$q$ is continuous at $F$. Therefore, it can only jump at vertical faces 
with $\lvert\nu_F(1)\rvert=1$ (otherwise there exists an edge $E$ that 
satisfies the above conditions). Figure~\ref{f:illuproofinfsup3} illustrates 
a vertical hyperplane with $\lvert\nu_F(1)\rvert=1$.

\textit{Case 2: $V_{h,D}=V_{\cF,D}$.}
In this case, define a function $v\in V_{0,\mathcal{M}}$ by 
$v=(0,v_2,0)$ with 
\begin{align}\label{e:defv2proofinfsup2}
v_2=\varphi_F. 
\end{align}
Then $v_2\vert_{F'}=0$ for all 
faces $F'\in\cF$ with $F'\neq F$. Since $q\in N_\mathcal{M}$ can only jump 
across vertical faces, an integration by parts proves 
\begin{align*}
 0=\int_{\Omega_\mathcal{M}} q\ddiv_h v\,dx
  = \int_{\Omega_\mathcal{M}} q\frac{\partial v_2}{\partial x_2}\,dx 
  = \int_F [q]_F v_2 \nu_F(2)\,ds 
  =\nu_F(2) [q]_F \int_F v_2\,ds.
\end{align*}
Since $\int_F v_2\,ds=\operatorname{area}(F)\neq 0$, 
this implies that $q$ is continuous at $F$ whenever $\nu_F(2)\neq 0$.

\medskip
\textit{In both cases,} the only situation where $q$ can jump is at vertical 
faces with $\lvert\nu_F(1)\rvert=1$, see Figure~\ref{f:illuproofinfsup3}
for an illustration.
Define $v=(v_1,0,0)\in V_{0,\mathcal{M}}$ by $v_1(z)=1$ and let 
\begin{align*}
  S:=\bigcup\{F'\mid F'\text{ is an interior face of }\mathcal{M} \text{ and }
\nu_F(1)=1\}.
\end{align*}
Since $q$ can only jump across $S$, an integration by parts then proves
\begin{align*}
  0=\int_{\Omega_\mathcal{M}} q\ddiv_h v\,dx 
   = \int_{\Omega_\mathcal{M}} q\frac{\partial v_1}{\partial x_1}\,dx 
   = \int_S [q]_S v_1 \nu_S(1)\,ds
   =  [q]_S\,  \int_S v_1\,ds\,.
\end{align*}
Since $\int_S v_1\,ds=\mathrm{area}(S)/3\neq 0$, this implies that 
$q$ is continuous on $\Omega_\mathcal{M}$.

Let $\mathcal{F}(z):=\{F\in\cF\mid z\in F\}$ 
denote the set of faces that share the vertex $z$.
The above argument proves that the two seminorms
\begin{align*}
 \rho_1(p_h)&:=\sqrt{\sum_{F\in\cF(z)} h_F \|[p_h]_F\|_{0,F}^2},\\
 \rho_2(p_h)&:=\sup_{v_h\in V_{0,\mathcal{M}}\setminus\{0\}} 
    \frac{\int_{\omega_z} p_h\div_h v_h\dx}{\|\nabla_h v\|_{0}}
\end{align*}
are equivalent on $Q_\mathcal{M}$. A scaling argument proves that the constant is 
independent of the mesh-size.
This proves a local inf-sup condition with respect to the (semi)-norm 
$\rho_1$. Assumption (H2) guarantees that the domain can 
be covered by the macroelements $\cM$. 
Then, \cite[Lemmas~1--4]{Stenberg1990} proves the global inf-sup condition
\begin{align*}
  \|p_h\|_ {0} 
    \lesssim \sup_{v\in V_{h,D}\setminus\{0\}} 
       \frac{b_h(v_h,p_h)}{\|\nabla_h v_h\|_{0}}
\end{align*}
with $p_h$ measured in the $L^2$ norm.

\step{2}
Let $(u_h,p_h)\in V_{h,D}\times Q_h$ be given and define 
for abbreviation
\begin{align*}
 B:=\sup_{(v_h,q_h)\in (V_{h,D}\times Q_h)\setminus\{0\}}
    \frac{\mathcal{B}(u_h,p_h;v_h,q_h)}{\ennorm{(v_h,q_h)}_h}.
\end{align*}
Step~1 guarantees the existence of $v_h\in V_{h,D}$ with 
$\|\nabla_h v_h\|_ {L^2(\Omega)}=1$ and
\begin{align*}
 \|p_h\|_ {0}\lesssim b_h(v_h,p_h) = \mathcal{B}(u_h,p_h;v_h,0) - a_h(u_h,v_h)
   \leq B + \|\nabla_h u_h\|_{0},
\end{align*}
which implies 
\begin{align*}
 \ennorm{(u_h,p_h)}_h^2
  \lesssim B^2 + \|\nabla_h u_h\|_ {0}^2.
\end{align*}
The discrete Korn inequality from Theorem~\ref{t:korn} implies 
\begin{align*}
 \|\nabla_h u_h\|_{0}
  \lesssim \|\varepsilon_h(u_h)\|_0.
\end{align*}
This implies,
\begin{align*}
 \|\nabla_h u_h\|_ {0}^2
 &\lesssim \|\varepsilon_h(u_h)\|_ {0}^2
   = \mathcal{B}(u_h,0;u_h,0)
        \leq B \;\ennorm{(u_h,p_h)}_h,
\end{align*}
and, therefore, 
\begin{align*}
 \ennorm{(u_h,p_h)}_h
  \lesssim B .
\end{align*}
This concludes the proof.
\end{proof}

\begin{remark}
The proof of Theorem~\ref{t:infsup} does not work in the case 
$V_{h,D}=\SCone\times \SCone\times \Snc$:
In this situation, the test functions $v_2$ that were defined 
in~\eqref{e:defv2proofinfsup} and in~\eqref{e:defv2proofinfsup2} have 
only one degree of freedom and, therefore, only the linear combination 
\begin{align*}
 \sum_{F\in\cF(\Omega_\mathcal{M}),F\text{ is vertical }}
   [q]_F \nu_F(2) \operatorname{area}(F) 
   = 0
\end{align*}
has to vanish, but the continuity on all $F$ with 
$\lvert\nu_F(2)\rvert<1$ cannot be concluded.
\end{remark}

From the discrete inf-sup condition from Theorem~\ref{t:infsup}, 
the discrete Korn inequality from Theorem~\ref{t:korn}, and the 
standard theory in mixed FEMs \cite{BoffiBrezziFortin2013},
we can immediately show the well-posedness of the problem which is stated 
in the following corollary.

\begin{corollary}\label{c:uniqexdisc}
Let $V_{h,D}=V_{2,D}$ or $V_{h,D}=V_{\mathcal{F},D}$.
There exists a unique solution $(u_h,p_h)\in V_{h,D}\times Q_h$
to~\eqref{StokesDisc} and it satisfies
\begin{align*}
 \|u_h\|_{1,h} + \|p_h\|_{L^2(\Omega)}
   \lesssim \|g\|_{L^2(\Omega)}.
\end{align*}
\end{corollary}

Recently, a new approach in the error analysis of non-conforming FEMs 
was introduced \cite{Gudi2010}. This approach employs techniques from 
the a~posteriori analysis to conclude a~priori results. 
This leads to a~priori error estimates that are independent of the regularity 
of the exact solution and that hold on arbitrary coarse meshes.
This approach was generalized by~\cite{HuMaShi2014} to the case of 
non-constant stresses.
The stability results of Theorem~\ref{t:korn} and~\ref{t:infsup} and the
abstract a~posteriori framework of~\cite{Carstensen05, CarstensenHu2007}
are the key ingredients in the following error estimate. 
The right-hand side of this error estimate includes oscillations of the 
right-hand side $g$, which are defined by 
\begin{align*}
  \operatorname{osc}(g,\tri):=\|g-\Pi_0 g\|_{L^2(\Omega)}, 
\end{align*}
where $\Pi_0$ denotes the $L^2$ projection to piecewise constant functions.
If $g$ is (piecewise) smooth, this term is of higher-order. 

\begin{theorem}[best-approximation error estimate]
\label{t:bestapprox}
Assume that Hypotheses (H1) and (H2) hold.
Let $(u,p)\in [H^1_D(\Omega)]^3 \times L^2(\Omega)$ be the exact solution 
to problem~\eqref{e:StokesWeak} and $(u_h,p_h)\in V_{h,D}\times Q_h$ be the 
discrete solution of~\eqref{StokesDisc} for $V_{h,D}=V_{2,D}$ or 
$V_{h,D}=V_{\mathcal{F},D}$. Then it holds 
\begin{align*}
  &\|u-u_h\|_{1,h} + \|p-p_h\|_{L^2(\Omega)}\\
  &\qquad
  \lesssim \inf_{v_h\in V_{h,D}} \|u-v_h\|_{1,h}
     + \inf_{q_h\in Q_h} \|p-q_h\|_{L^2(\Omega)}
     +\|\varepsilon(u)-\Pi_0\varepsilon(u)\|_{L^2(\Omega)}
     +\operatorname{osc}(g,\tri).
\end{align*}
\end{theorem}

\begin{proof}
The proof is in the spirit of~\cite{Gudi2010,BadiaCodinaGudiGuzman2014} 
and the generalization of~\cite{HuMaShi2014}. The outline of 
the proof is included for completeness.

Let $(w_h,r_h)\in V_{h,D}\times Q_h$ be arbitrary. The inf-sup condition 
of Theorem~\ref{t:infsup} guarantees the existence of 
$(v_h,q_h)\in V_{h,D}\times Q_h$ with $\ennorm{(v_h,q_h)}=1$ and
\begin{align*}
 \|u_h-w_h\|_{1,h} + \|p_h-r_h\|_{L^2(\Omega)}
 \lesssim \mathcal{B}(u_h-w_h,p_h-r_h;v_h,q_h).
\end{align*}
Let $E_h:V_{h,D}\to [H^1_D(\Omega)]^3$ be the operator that is the identity 
in the first two components and an averaging (enriching) operator that maps 
$\Snc$ to $\SCone$ in the third component, 
see \cite{Gudi2010} for details.
As $(u,p)$ is the exact solution and $(u_h,p_h)$ is the discrete solution,
this implies 
\begin{align*}
 &\mathcal{B}(u_h-w_h,p_h-r_h;v_h,q_h)\\
  &\qquad= (g,v_h-E_h v_h)_{L^2(\Omega)} 
     +\mathcal{B}(u-w_h,p-r_h;E_h v_h,q_h)
     -\mathcal{B}(w_h,r_h;v_h - E_h v_h, 0)
     .
\end{align*}
A Cauchy inequality and the stability of the enriching operator $E_h$ prove 
\begin{align*}
 \lvert\mathcal{B}(u-w_h,p-r_h;E_h v_h,q_h)\rvert
   \leq \|u-w_h\|_{1,h} + \|p-r_h\|_{L^2(\Omega)} .
\end{align*}
Let $\tau_h:=2\mu\varepsilon_h(w_h)-r_h I_{3\times 3}$ denote the stress-like 
variable for $w_h$ and $r_h$.
A piecewise integration by parts proves for the remaining 
terms
\begin{equation}\label{e:bestapprox1}
\begin{aligned}
 &\lvert (g,v_h-E_h v_h)_{L^2(\Omega)} 
     +\mathcal{B}(w_h,r_h;v_h - E_h v_h, 0)\rvert\\
  &\qquad\qquad 
  = (g-\ddiv_h\tau_h,v_h-E_h v_h)_{L^2(\Omega)} 
    + \sum_{F\in\mathcal{F}} \int_{F} [(v_h-E_h v_h)\cdot \tau_h\nu_F]_F\,ds .
\end{aligned}
\end{equation}
The first term on the right-hand side is estimated with the help of a Cauchy 
inequality and the approximation properties of $E_h$ \cite{Gudi2010}
\begin{align*}
 (g-\ddiv_h\tau_h,v_h-E_h v_h)_{L^2(\Omega)} 
 \lesssim \|h_\tri (g-\ddiv_h\tau_h)\|_{L^2(\Omega)}.
\end{align*}
This is a standard a~posteriori error estimator 
term \cite{Carstensen05, CarstensenHu2007} and the bubble function 
technique~\cite{Verfuerth2013} proves the efficiency 
\begin{align*}
 \|h_\tri (g-\ddiv_h\tau_h)\|_{L^2(\Omega)}
   \lesssim \|u-w_h\|_{1,h} + \|p-r_h\|_{L^2(\Omega)} 
      + \operatorname{osc}(g,\tri).
\end{align*}
Let $\langle \bullet\rangle_F:=(\bullet\vert_{T_+} +\bullet\vert_{T_-})/2$ 
denote the average along $F=T_+\cap T_-$.
An elementary calculation proves for any $\alpha$ and $\beta$ that 
$[\alpha\beta_F]=\langle \alpha\rangle_F [\beta]_F 
+[\alpha]_F\langle\beta\rangle_F$.
We employ this identity for the second term of the right-hand side 
of~\eqref{e:bestapprox1} and conclude 
\begin{align*}
 &\int_{F} [(v_h-E_h v_h)\cdot \tau_h\nu_F]_F\,ds\\
 &\qquad\qquad
   =\int_{F} [v_h-E_h v_h]_F\cdot \langle\tau_h\nu_F\rangle_F\,ds
      + \int_F \langle v_h-E_h v_h\rangle_F \cdot [\tau_h\nu_F]_F\,ds.
\end{align*}
The second term on the right-hand side is again estimated with a~posteriori 
techniques \cite{Carstensen05,CarstensenHu2007,Verfuerth2013} which results in 
\begin{align*}
 \sum_{F\in\mathcal{F}} \int_F \langle v_h-E_h v_h\rangle_F \cdot [\tau_h\nu_F]_F\,ds
   &\lesssim \bigg(\sum_{F\in\mathcal{F}} h_F 
         \|[\tau_h]_F\nu_F\|_{L^2(F)}^2 \bigg)^{1/2}\\
   &\lesssim \|u-w_h\|_{1,h} + \|p-r_h\|_{L^2(\Omega)} 
      + \operatorname{osc}(g,\tri).
\end{align*}
Since $[v_h-E_h v_h]_F$ is affine on $F=T_+\cap T_-$ and vanishes at the midpoint 
of $F$, we conclude for the first term as in~\cite{HuMaShi2014} that 
\begin{align*}
 &\int_{F} [v_h-E_h v_h]_F\cdot \langle\tau_h\nu_F\rangle_F\,ds
   = \int_F [v_h-E_h v_h]_F\cdot \langle (1-\Pi_0)\tau_h\nu_F\rangle_F\,ds\\
 &\qquad\qquad \qquad\qquad 
  = \frac{1}{2}\int_F [v_h-E_h v_h]_F\cdot ( (1-\Pi_0)\tau_h\vert_{T_+} + 
      (1-\Pi_0)\tau_h\vert_{T_-})\nu_F\,ds.
\end{align*}
Note that $\tau_h=2\mu\varepsilon_h(w_h)-r_h I_{3\times 3}$ and 
$r_h$ is piecewise constant.
Therefore, trace inequalities \cite{BrennerScott08} and an inverse inequality 
imply that this is bounded by 
\begin{align*}
 &\frac{1}{2}\int_F [v_h-E_h v_h]_F\cdot ( (1-\Pi_0)\tau_h\vert_{T_+} + 
      (1-\Pi_0)\tau_h\vert_{T_-})\nu_F\,ds\\
 &\qquad\qquad      
  \lesssim \|(1-\Pi_0) \tau_h\|_{L^2(T_+\cup T_-)}\\
  &\qquad\qquad      
  \leq \|(1-\Pi_0)\varepsilon_h(u-w_h)\|_{L^2(T_+\cup T_-)}
    + \|(1-\Pi_0)\varepsilon (u)\|_{L^2(T_+\cup T_-)}.
\end{align*}
Since $\|(1-\Pi_0)\varepsilon_h(u-w_h)\|_{L^2(T_+\cup T_-)}
\leq \|\varepsilon_h(u-w_h)\|_{L^2(T_+\cup T_-)}$, the combination 
of the foregoing inequalities and the finite overlap of the patches 
conclude the proof.
\end{proof}

\section{Counterexamples for $P_1$-$P_0$ discretization}
\label{s:counterexamples}

The following two counterexamples prove that the inf-sup condition
for the ansatz space 
$\SCone\times \SCone\times \Snc$ 
cannot hold in general, as well as that 
a discrete Korn inequality for the ansatz 
space $\SCone\times \Snc\times \Snc$
is not satisfied in general.

\subsection{Instability of 
$\SCone\times \SCone\times \Snc$}
\label{ss:noinfsup}

The following counterexample proves that the inf-sup condition
\begin{align}\label{e:infsupcounterexample}
  \|p_h\|_0 \lesssim \sup_{v_h\in V_{h,D}\setminus\{0\}}
     \frac{b_h(v_h,p_h)}{\|\nabla_h v_h\|_0}
   \qquad\text{for all }p_h\in Q_h
\end{align}
is not 
fulfilled for functions in 
$\SCone\times \SCone\times \Snc$. 

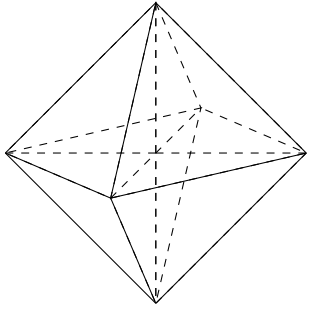
\begin{figure}
\begin{center}
 \begin{tikzpicture}[scale=2]
   \draw[dashed] (-1,0)--(0,0)--(0,-1);
   \draw[dashed] (0,0)--(0,1)--(-1,0);
   \draw[dashed] (1,0)--(0,0)--(0,-1)--cycle;
   \draw[dashed] (0,0)--(0,1)--(1,0);
   \draw[dashed] (-0.3,-0.3)--(0.3,0.3)--(-1,0)--(-0.3,-0.3)--(1,0)--(0.3,0.3);
   \draw[dashed] (0,1)--(0.3,0.3)--(0,-1)--(-0.3,-0.3)--cycle;
   \draw (0,-1)--(-1,0);
   \draw (0,1)--(-1,0);
   \draw (1,0)--(0,-1);
   \draw (0,1)--(1,0);
   \draw (-1,0)--(-0.3,-0.3)--(1,0);
   \draw (0,-1)--(-0.3,-0.3)--cycle;
   \draw (0,1)--(-0.3,-0.3);
 \end{tikzpicture}
\end{center}
\caption{\label{f:triangulationcounterex}Triangulation 
   of Subsections~\ref{ss:noinfsup} and~\ref{ss:noKorn42NC}.}
\end{figure}

Consider the node $z:=(0,0,0)$ with nodal patch
\begin{equation}\label{e:triangulationcounterex}
\begin{aligned}
 \cT&:=\{T_{jk\ell}\mid (j,k,\ell)\in\{1,2\}^3\} \\
 \text{with}\quad
 T_{jk\ell}&:=\mathrm{conv}\{z,(-1)^j e_1,(-1)^k e_2,(-1)^\ell e_3\}
\end{aligned}
\end{equation}
with the unit vectors $e_1$, $e_2$ and $e_3$; see Figure~\ref{f:triangulationcounterex}
for an illustration.
Let $\Omega:=\bigcup\tri$ be the corresponding domain with pure Dirichlet 
boundary $\Gamma_D=\partial\Omega$.
Define the function $q\in P_0(\tri)$ by 
\begin{align*}
 \tilde{q}\vert_{T_{111}}:=\tilde{q}\vert_{T_{112}}:=1,
 \qquad 
 \tilde{q}\vert_{T_{121}}:=\tilde{q}\vert_{T_{122}}:=0,\\
 \tilde{q}\vert_{T_{211}}:=\tilde{q}\vert_{T_{212}}:=0,
 \qquad 
 \tilde{q}\vert_{T_{221}}:=\tilde{q}\vert_{T_{222}}:=1
\end{align*}
and $q=\tilde{q}-\int_\Omega \tilde{q}\,dx$.
The normal vectors to the following intersections read:
\begin{align*}
 \text{for }E_1:=(T_{111}\cup T_{112})\cap (T_{121}\cup T_{122})
 \text{ define }\nu_1=(0,1,0),\\
 \text{for }E_2:=(T_{111}\cup T_{112})\cap (T_{211}\cup T_{212})
 \text{ define }\nu_2=(1,0,0),\\
 \text{for }E_3:=(T_{221}\cup T_{222})\cap (T_{121}\cup T_{122})
 \text{ define }\nu_2=(1,0,0),\\
 \text{for }E_4:=(T_{221}\cup T_{222})\cap (T_{211}\cup T_{212})
 \text{ define }\nu_2=(0,1,0).
\end{align*}
Let $v_h=(v_1,v_2,v_3)\in \SCone\times \SCone\times \Snc$. 
Since $v_1$ and $v_2$ have only one degree of freedom, it follows
$v_1(\mathrm{mid}(F))=v_1(\mathrm{mid}(F'))$ for all faces $F,F'$. 
An integration by parts then proves
\begin{align*}
 \int_{\Omega_z} q \frac{\pa v_1}{\pa x_1}\dx
  = \int_{E_2} [q]_{E_2} v_1 \ds 
    + \int_{E_3} [q]_{E_3} v_1\ds
  = 2 v_1(\mathrm{mid}(F)) ([q]_{E_2}+[q]_{E_3})=0
\end{align*}
and similarly 
\begin{align*}
 \int_{\Omega_z} q \frac{\pa v_2}{\pa x_2}\dx=0.
\end{align*}
Since the third components of all of the normal vectors for faces, where 
$q$ jumps, vanish, a further 
integration by parts leads to 
\begin{align*}
 \int_{\Omega_z} q \frac{\pa v_3}{\pa x_3}\dx=0.
\end{align*}
The sum of the last three equalities yields
\begin{align*}
 \int_{\Omega_z} q \div v_h\dx=0.
\end{align*}
Since $v_{h}\in V_{h,D}$ is arbitrary, this proves that the 
inf-sup condition~\eqref{e:infsupcounterexample} cannot hold, and, hence, 
a discretisation with the space 
$\SCone\times \SCone\times \Snc$ is not stable.

\subsection{Instability of 
$\SCone\times \Snc\times \Snc$}
\label{ss:noKorn42NC}

The following counterexamples prove that there are functions
in $\SCone\times \Snc\times \Snc$ such that 
$\varepsilon_h(\bullet)$
vanishes, but which are not global rigid body motions. This proves that a 
Korn inequality cannot hold on $\SCone\times \Snc\times \Snc$.
The first part illustrates, how a missing Korn inequality for the 
discretisation $\Snc\times\Snc$ in the two dimensional situation generalizes 
to the three dimensional case, while the second part proves 
that there exists arbitrary fine meshes, such that a counterexample can be 
constructed.

For the first counterexample, consider first the two dimensional square 
$\widetilde{\Omega}=(-1,1)^2$ with the two dimensional triangulation 
$\tri_{2D}=\{\widetilde{T}_1,\widetilde{T}_2,\widetilde{T}_3,\widetilde{T}_4\}$ 
with triangles $\widetilde{T}_j$ as in 
Figure~\ref{f:2dtriangcounter}.
\begin{figure}
\subfloat[\label{f:2dtriangcounter}2D triangulation]{
 \begin{tikzpicture}
   \draw (-2,-2)--(2,2)--(-2,2)--(2,-2);
   \draw (-2,2)--(-2,-2)--(2,-2)--(2,2);
   \node () at (0,1){$\widetilde{T}_1$};
   \node () at (-1,0){$\widetilde{T}_2$};
   \node () at (0,-1){$\widetilde{T}_3$};
   \node () at (1,0){$\widetilde{T}_4$};
   
    \phantom{\draw[dashed,cyan, thick] (0,-2.5);}
 \end{tikzpicture}
 }
\subfloat[\label{f:2dcounter}2D counterexample]{
 \begin{tikzpicture}[scale=2]
    \coordinate (M) at (0,0);
 \coordinate (A) at (1,-1);
 \coordinate (B) at (1,1);
 \coordinate (C) at (-1,1);
 \coordinate (D) at (-1,-1);

 \coordinate (mA) at ($ (M)!.5!(A) $);
 \coordinate (mB) at ($ (M)!.5!(B) $);
 \coordinate (mC) at ($ (M)!.5!(C) $);
 \coordinate (mD) at ($ (M)!.5!(D) $);

 \draw (A)--(C) (B)--(D);
 \draw (A)--(B)--(C)--(D)--cycle;
 
 \draw[fill]  (mA) circle(.04);
 \draw[fill]  (mB) circle(.04);
 \draw[fill]  (mC) circle(.04);
 \draw[fill]  (mD) circle(.04);

 \draw[->,ultra thick] (mA)--($ (M)!.75!(A) $);
 \draw[->,ultra thick] (mB)--($ (M)!.25!(B) $);
 \draw[->,ultra thick] (mC)--($ (M)!.75!(C) $);
 \draw[->,ultra thick] (mD)--($ (M)!.25!(D) $);

 \coordinate (T1a) at ($(A)+(.25,0)$);
 \coordinate (T1b) at ($(B)+(-.25,0)$);
 \coordinate (T1m) at ($(M)+(0,-.25)$);

 \coordinate (T2b) at ($(B)+(0,-.25)$);
 \coordinate (T2c) at ($(C)+(0,.25)$);
 \coordinate (T2m) at ($(M)+(-.25,0)$);
 
 \coordinate (T3c) at ($(C)+(-.25,0)$);
 \coordinate (T3d) at ($(D)+(.25,0)$);
 \coordinate (T3m) at ($(M)+(0,.25)$);
 
 \coordinate (T4d) at ($(D)+(0,.25)$);
 \coordinate (T4a) at ($(A)+(0,-.25)$);
 \coordinate (T4m) at ($(M)+(.25,0)$);
 
 \draw[dashed, thick] (T1a)--(T1b)--(T1m)--cycle;
 \draw[dashed,thick] (T2b)--(T2c)--(T2m)--cycle;
 \draw[dashed,thick] (T3c)--(T3d)--(T3m)--cycle;
 \draw[dashed,thick] (T4d)--(T4a)--(T4m)--cycle;
 \end{tikzpicture}
 }
 \subfloat[\label{f:3dcounterCRCR1}Generalization to 3D]{
 \begin{tikzpicture}[scale=0.9]
 \coordinate (M) at (0,0);
 \coordinate (A) at (1,-1);
 \coordinate (B) at (3,1);
 \coordinate (C) at (-1,1);
 \coordinate (D) at (-3,-1);
 \coordinate (P) at (-.2,5);
 
 \coordinate (dA) at ($ (M)!.33!(A) $);
 \coordinate (dB) at ($ (M)!.27!(B) $);
 \coordinate (dC) at ($ (M)!.33!(C) $);
 \coordinate (dD) at ($ (M)!.27!(D) $);

 \draw[thick] (D)--(A)--(B);
 \draw[dashed,thick] (B)--(C)--(D);
 \draw (B)--(D)  (A)--(C);
 
 \draw (A)--(P)--(B)  (D)--(P);
 \draw[dashed] (C)--(P);
 \draw[dashed] (M)--(P);
 
 \draw[fill]  (dA) circle(.06);
 \draw[fill]  (dB) circle(.06);
 \draw[fill]  (dC) circle(.06);
 \draw[fill]  (dD) circle(.06);
 
 \draw[->,ultra thick] (dA)--($ (M)!.66!(A) $);
 \draw[->,ultra thick] (dB)--($ (M)!.11!(B) $);
 \draw[->,ultra thick] (dC)--($ (M)!.66!(C) $);
 \draw[->,ultra thick] (dD)--($ (M)!.11!(D) $);
 
 \draw[fill,gray]  ($ (M)!.33!(A)!.25!(P) $) circle(.1);
 \draw[fill,gray]  ($ (M)!.44!(B)!.25!(P) $) circle(.1);
 \draw[fill,gray]  ($ (M)!.33!(C)!.25!(P) $) circle(.1);
 \draw[fill,gray]  ($ (M)!.44!(D)!.25!(P) $) circle(.1);

 \coordinate (T1a) at ($(A)+(.15,0)$);
 \coordinate (T1b) at ($(B)+(-.2,.05)$);
 \coordinate (T1m) at ($(M)+(-.05,-.1)$);
 \coordinate (T1p) at ($(P)+(-.05,-.1)$);
 
 \draw[dashed,ultra thick] (T1a)--(T1b)--(T1m)--cycle;
 \draw[dashed,ultra thick] (T1a)--(T1p)--(T1b)--(T1p)--(T1m)--(T1p);
 
\end{tikzpicture}
}
\caption{\label{f:counterCRCR1}Generalization of 2D counterexample
 for the instability of $\SCone\times\Snc\times\Snc$.}
\end{figure}
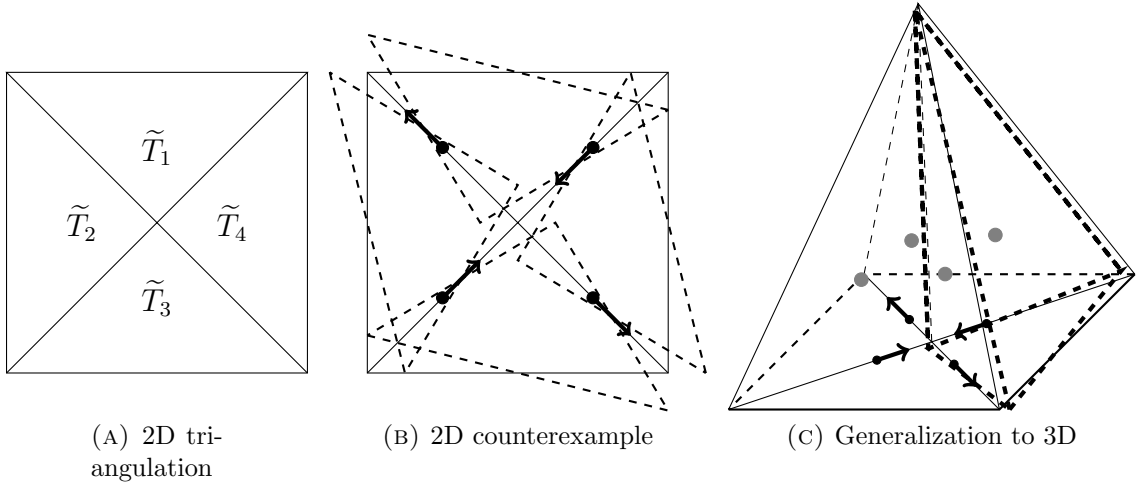
Then a piecewise rigid body motion $\mathrm{RM}_\mathrm{2D}$ 
that is continuous at the midpoints 
of the (2D) edges of the triangulation is depicted in Figure~\ref{f:2dcounter}.
This counterexample is also given in~\cite[Sect.~5]{FalkMorley1990}
and \cite{Arnold1993} to prove 
that there are 2D triangulations where Korn's inequality does not hold,
even if boundary conditions are imposed.
Consider now the triangulation $\tri:=\{T_1,T_2,T_3,T_4\}$ in 3D with 
$T_j:=\operatorname{conv}\{(1,0,0),\widetilde{T}_j\}$, where 
$\widetilde{T}_j$ from above is considered as a set in the plain 
$\{0\}\times \R^2$.
Shifting the continuity points of $\mathrm{RM}_\mathrm{2D}$, 
such that the function is continuous at 
$(-1/3,-1/3),(-1/3,1/3),(1/3,-1/3),(1/3,1/3)$, a 
piecewise rigid body motion with respect to $\tri$ is given by 
$\mathrm{RM}_\mathrm{3D}(x):=(0,\mathrm{RM}_\mathrm{2D}(x_2,x_3))$.
Since it is continuous at the points 
$(0,-1/3,-1/3),(0,-1/3,1/3),(0,1/3,-1/3),(0,1/3,1/3)$
and constant in $x$-direction, it is also continuous at the midpoints 
of the interior faces. This proves that a Korn inequality cannot 
hold for the space $\SCone\times \Snc\times \Snc$.

For the second part, let the triangulation $\widehat{\tri}$ be 
given by $\widehat{\tri}:=\{\widehat{T}_1,\widehat{T}_2,\widehat{T}_3,
\widehat{T}_4\}$ with the tetrahedra
\begin{align*}
 \widehat{T}_1&:=\operatorname{conv}\{(0,0,0),(1,0,0),(0,1,0),(0,0,1)\},\\
 \widehat{T}_2&:=\operatorname{conv}\{(0,0,0),(1,0,0),(0,-1,0),(0,0,1)\},\\
 \widehat{T}_3&:=\operatorname{conv}\{(0,0,0),(1,0,0),(0,-1,0),(0,0,-1)\},\\
 \widehat{T}_4&:=\operatorname{conv}\{(0,0,0),(1,0,0),(0,1,0),(0,0,-1)\};\\
\end{align*}
see Figure~\ref{f:triang_counterexample3} for an illustration.
Let $a\in\R$ be arbitrary. Define a piecewise rigid body motion $\varphi$ by 
\begin{align*}
 \varphi_{\widehat{T}_1}(x)&:= (0,a-3a x_3,-a+3a x_2),\\
 \varphi_{\widehat{T}_4}(x)&:= (0,-a+3a x_3,-a-3a x_2),\\
 \varphi_{\widehat{T}_4}(x)&:= (0,-a-3a x_3,a+3a x_2),\\
 \varphi_{\widehat{T}_4}(x)&:= (0,a+3a x_3,a-3a x_2),\\
\end{align*}
This function is continuous at the interior face's midpoints and vanishes 
at the midpoints of the boundary faces. Therefore it can be extended by zero 
to the rectangle $(0,1)\times(-1,1)\times (-1,1)$. As those functions can be 
easily glued together, this proves that even for arbitrary fine mesh-sizes, 
there exists piecewise rigid body motions in $\SCone\times \Snc\times \Snc$.

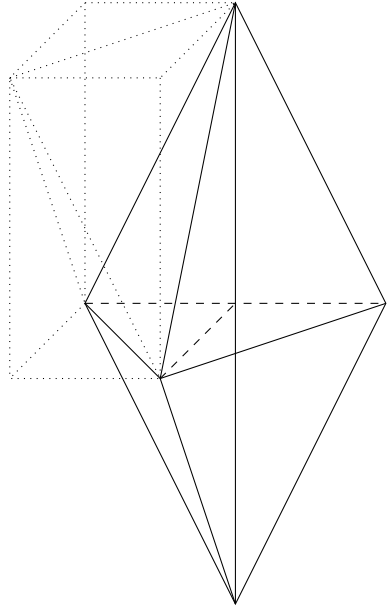
\begin{figure}
 \begin{center}
  \begin{tikzpicture}[scale=2]
    \draw (0,-2)--(0,2);
    \draw (0,2)--(-0.5,-0.5)--(0,-2);
    \draw[dashed] (-1,0)--(1,0);
    \draw (0,2)--(-1,0)--(0,-2)--(1,0)--cycle;
    \draw (-1,0)--(-0.5,-0.5)--(1,0);
    \draw[dashed] (-0.5,-0.5)--(0,0);
    \draw[dotted] (-1,0)--(-1.5,-0.5)--(-0.5,-0.5)--(-1.5,1.5)--(-1.5,-0.5);
    \draw[dotted] (-1.5,1.5)--(-1,0);
    \draw[dotted] (-1.5,1.5)--(-0.5,1.5)--(0,2)--(-1,2)--(-1.5,1.5)--(0,2);
    \draw[dotted] (-1,2)--(-1,0);
    \draw[dotted] (-0.5,1.5)--(-0.5,-0.5);
  \end{tikzpicture}
 \end{center}
 \caption{\label{f:triang_counterexample3}Triangulation $\widehat{\tri}$ from 
  the second counterexample in Section~\ref{ss:noKorn42NC}. The dotted lines 
  indicate a possible extension to the extension of this triangulation 
  to the rectangle $(0,1)\times (-1,1)\times (-1,1)$.}
\end{figure}

\section{Numerical experiments}
\label{s:numerics}

This section compares the performance of the two suggested discretisations from 
Section~\ref{s:FEM} and the conforming Bernardi-Raugel FEM in 
numerical experiments. The Bernardi-Raugel FEM~\cite{BernardiRaugel1985} is a conforming 
FEM that approximates the velocity in the space
\begin{align*}
 V_\mathrm{BR}:=\left\{v_h\in (H^1_D(\Omega))^3\left|
  \begin{array}{l}
    \exists v_\mathrm{C}\in (S^{\mathrm{C},1})^3\text{ and }
    \forall F\in \cF\, \exists \alpha_F\in \R
    \text{ such that }\\
    v_h= v_\mathrm{C}+\sum_{F\in\cF\setminus\cF(\Gamma_D)}\alpha_F\varphi_F\nu_F
  \end{array}
    \right\}\right. ,
\end{align*}
where $\nu_F$ denotes the normal for a face $F$ and 
$\varphi_F$ denotes the face bubble defined in Section~\ref{s:FEM}.
The pressure is approximated in $Q_h$.
The errors of the different methods are compared in the following 
subsections, while the computational effort of the three different methods 
is illustrated in Tables~\ref{tab:ndof}--\ref{tab:nnz} in terms of the 
number of non-zero entries of the system matrices and the number of degrees of freedom.
The number of degrees of freedom are lower for the Bernardi-Raugel 
method compared to the two proposed methods. 
However, since the support of the face bubble functions in 
$V_{\mathcal{F},D}$ consists of only two tetrahedra, the number of non-zero entries 
of the system matrices of the Bernardi-Raugel FEM and the proposed FEM 
with $V_{h,D}=V_{\mathcal{F},D}$ shows only slight differences.

\begin{table}
\begin{tabular}{l|ccc|ccc}
 level of & \multicolumn{3}{c|}{cube} & \multicolumn{3}{c}{L-shaped domain}\\
 refinement & KS bubbles & KS $P_2$ & BR & KS bubbles & KS $P_2$ & BR\\
 \hline
  1 & 1.7e+02 & 1.3e+02 & 1.0e+02 & 5.1e+02 & 3.9e+02 & 3.2e+02\\
  2 & 1.5e+03 & 1.2e+03 & 9.7e+02 & 4.5e+03 & 3.6e+03 & 2.9e+03\\
  3 & 1.2e+04 & 1.0e+04 & 8.3e+03 & 3.8e+04 & 3.1e+04 & 2.5e+04\\
  4 & 1.0e+05 & 8.8e+04 & 6.9e+04 & 3.1e+05 & 2.6e+05 & 2.0e+05\\
  5 & 8.6e+05 & 7.2e+05 & 5.6e+05 \\
\end{tabular}
\caption{\label{tab:ndof}Number of degrees of freedom in the 
 numerical experiments in Sections~\ref{ss:numericsCube1}--\ref{ss:numericsLG3}.}
\end{table}

\begin{table}
\begin{tabular}{l|ccc|ccc}
 level of & \multicolumn{3}{c|}{cube} & \multicolumn{3}{c}{L-shaped domain}\\
 refinement & KS bubbles & KS $P_2$ & BR & KS bubbles & KS $P_2$ & BR\\
 \hline
 1 & 5.4e+03 & 6.2e+03 & 5.3e+03 & 1.5e+04 & 1.7e+04 & 1.5e+04 \\
  2 & 4.0e+04 & 4.4e+04 & 3.9e+04 & 1.1e+05 & 1.3e+05 & 1.1e+05 \\
  3 & 3.1e+05 & 3.4e+05 & 3.0e+05 & 9.3e+05 & 1.0e+06 & 8.9e+05 \\
  4 & 2.4e+06 & 2.7e+06 & 2.3e+06 & 7.4e+06 & 8.0e+06 & 7.0e+06 \\
  5 & 1.9e+07 & 2.1e+07 & 1.8e+07 \\
\end{tabular}
\caption{\label{tab:nnz}Number of non-zero entries in the system matrix in the 
 numerical experiments in Sections~\ref{ss:numericsCube1}--\ref{ss:numericsLG3}.}
\end{table}

\subsection{Smooth solution on the cube, I}
\label{ss:numericsCube1}

This subsection considers the smooth solution 
\begin{align*}
 u(x) &= \begin{pmatrix}
          \pi\cos(\pi x_2)\sin(\pi x_1)^2 \sin(\pi x_2)\sin(\pi x_3)\\
          -\pi\cos(\pi x_1) \sin(\pi x_2)^2 \sin(\pi x_1)\sin(\pi x_3)\\
          0
        \end{pmatrix},\\
  p(x) &= 0
\end{align*}
on the Cube $\Omega=(0,1)^3$ with Neumann boundary $\Gamma_N=(0,1)^2\times\{0\}$
and Dirichlet boundary $\Gamma_D=\partial\Omega\setminus\Gamma_N$. 
The solutions to~\eqref{StokesDisc} with 
$V_{h,D}=V_{2,D}$ and $V_{h,D}=V_{\mathcal{F},D}$ and the solution for the 
Bernardi-Raugel FEM for the right-hand 
side $f$ and $g$ given by the exact solution
are computed on a 
sequence of red-refined triangulations. The initial triangulation is 
depicted in Figure~\ref{f:inittriangcube}.
\begin{figure}
 \begin{center}
   \includegraphics[width=0.3\textwidth]{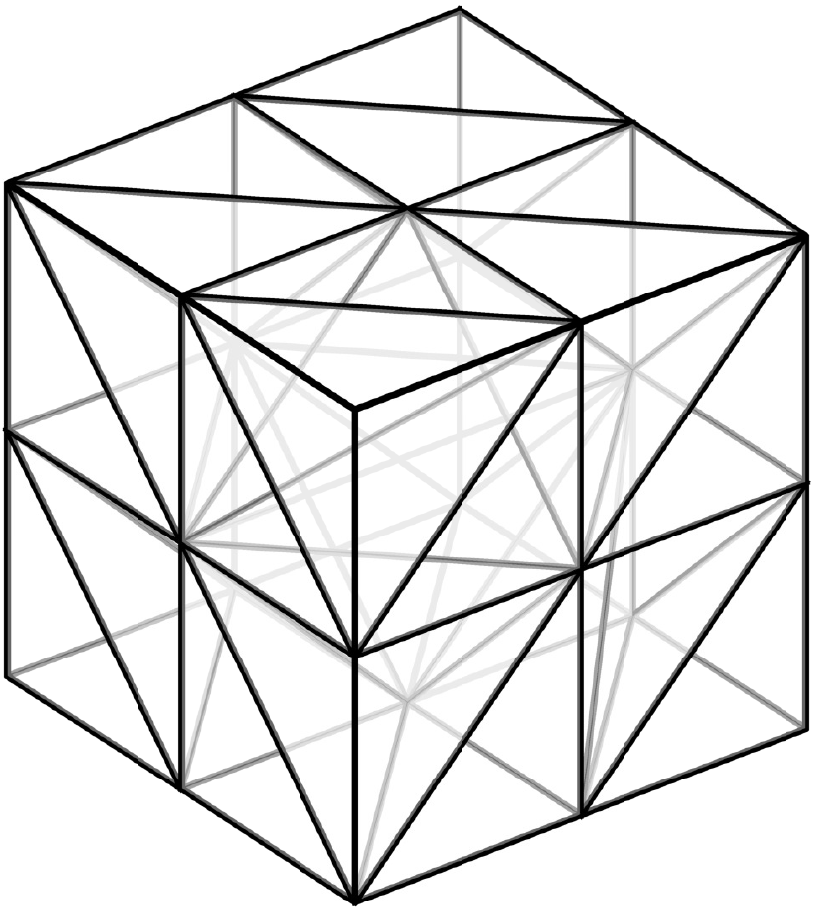}
   \hspace{2cm}
   \includegraphics[width=0.3\textwidth]{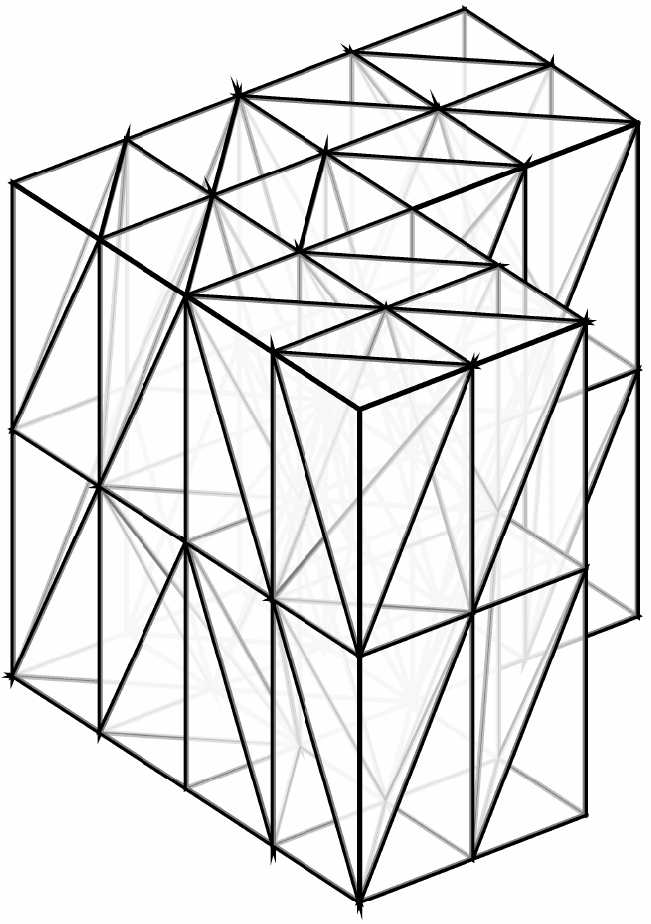}
 \end{center}
\caption{\label{f:inittriangcube}Initial triangulations of the cube and the 
tensor product L-shaped domain.}
\end{figure}
The $H^1$ errors and the $L^2$ errors of the velocity $u$ are depicted 
in the convergence history plot in Figure~\ref{f:cube1}. 
The $H^1$ errors show convergence rates of $\mathcal{O}(h)$ for all 
methods, while the convergence rates of the $L^2$ errors of all methods are near
$\mathcal{O}(h^2)$ with a slightly larger convergence rate for $V_{2,D}$.
\begin{figure}
 \begin{center}
   \includegraphics[width=0.95\textwidth]{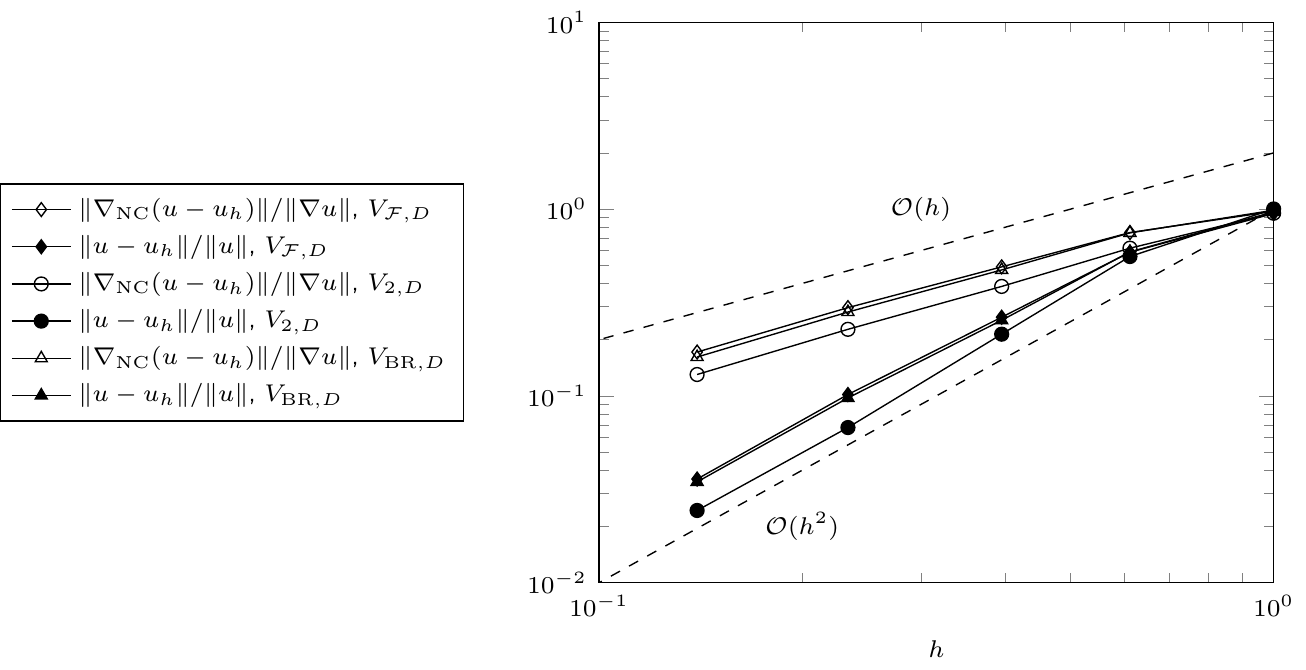}
 \end{center}
\caption{\label{f:cube1}Convergence history plot for the example from 
Subsection~\ref{ss:numericsCube1}}
\end{figure}

\subsection{Smooth solution on the cube, II}
\label{ss:numericsCube2}

This subsection considers the smooth exact solution
\begin{align*}
 u&(x) = \begin{pmatrix}
          10x_1 x_2^4+10x_1 x_3^4-4 x_1^5\\
          10 x_2 x_1^4+10 x_2 x_3^4 -4x_2^5\\
          10 x_3 x_1^4+10 x_3 x_2^4 - 4x_3^5
        \end{pmatrix},\\
  p&(x) = -60 x_1^2 x_2^2-60x_1^2 x_3^2 -60 x_2^2 x_3^2 +20 x_1^4 + 20x_2^4+20x_3^4
\end{align*}
on the cube $\Omega=(-1,1)^3$ with Neumann boundary $\Gamma_N=(0,1)^2\times\{-1\}$
and Dirichlet boundary $\Gamma_D=\partial\Omega\setminus\Gamma_N$. 
As in subsection~\ref{ss:numericsCube1}, 
the solutions to~\eqref{StokesDisc} with 
$V_{h,D}=V_{2,D}$ and $V_{h,D}=V_{\mathcal{F},D}$ and the solution 
of the Bernardi-Raugel FEM for the right-hand 
side $f$ and $g$ given by the exact solution
are computed on a 
sequence of red-refined triangulations. The initial triangulation is 
depicted in Figure~\ref{f:inittriangcube}.
The $H^1$ errors and the $L^2$ errors of the velocity $u$ are depicted 
in the convergence history plot in Figure~\ref{f:cube2}. 
The $H^1$ errors show convergence rates of $\mathcal{O}(h)$ for all 
methods, while the convergence rates of the $L^2$ errors of both nonconforming 
methods are slightly worse than $\mathcal{O}(h^2)$. The convergence rate 
of the $L^2$ error for the Bernardi-Raugel FEM seems to be larger than that of 
the two nonconforming FEMs.
\begin{figure}
 \begin{center}
   \includegraphics[width=0.95\textwidth]{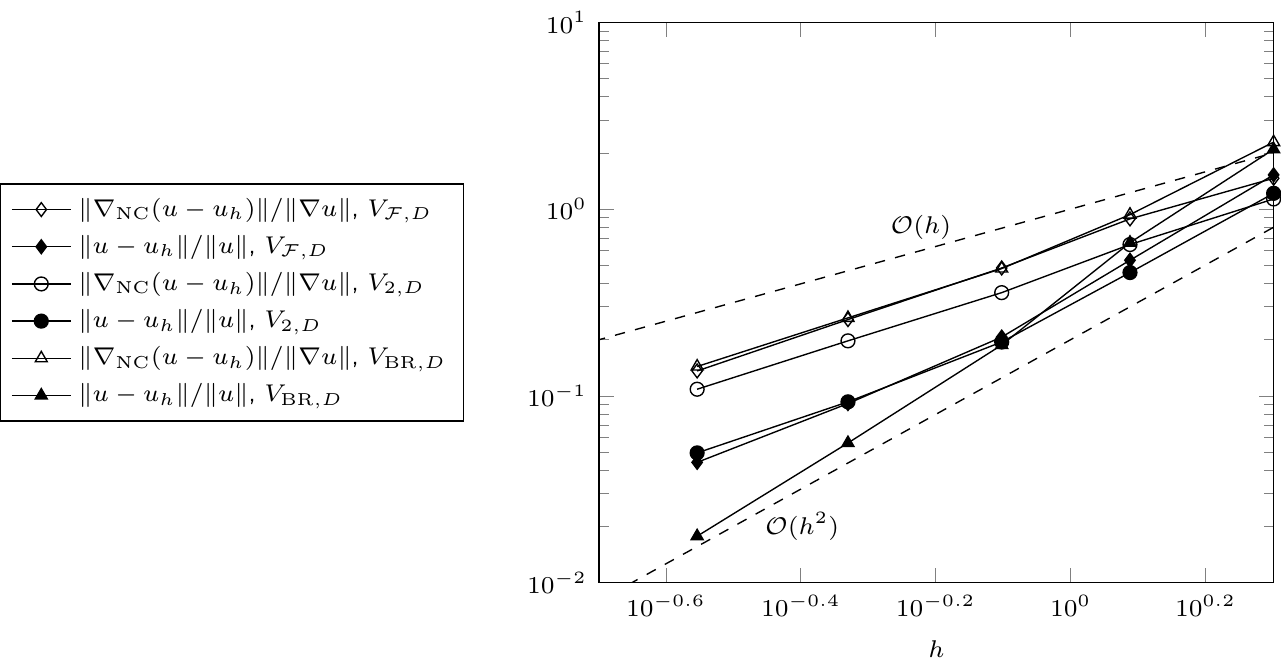}
 \end{center}
 \caption{\label{f:cube2}Convergence history plot for the example from 
Subsection~\ref{ss:numericsCube2}}
\end{figure}

\subsection{Smooth solution on the cube, III}
\label{ss:numericsCube3}

This subsection considers the smooth exact solution
\begin{align*}
 u(x) &= \begin{pmatrix}
          2x_2 x_3 (x_1^2-1)^2(x_2^2-1)(x_3^2-1)\\
          -x_1 x_3 (x_1^2-1)(x_2^2-1)^2(x_3^2-1)\\
          -x_1 x_2 (x_1^2-1)(x_2^2-1)(x_3^2-1)^2
        \end{pmatrix},\\
  p(x) &= x_1 x_2 x_3
\end{align*}
on the cube $\Omega=(-1,1)^3$ with Neumann boundary $\Gamma_N=(0,1)^2\times\{-1\}$
and Dirichlet boundary $\Gamma_D=\partial\Omega\setminus\Gamma_N$. 
As in subsections~\ref{ss:numericsCube1}--\ref{ss:numericsCube2}, 
the solutions to~\eqref{StokesDisc} with 
$V_{h,D}=V_{2,D}$ and $V_{h,D}=V_{\mathcal{F},D}$ and the solution of the 
Bernardi-Raugel FEM for the right-hand 
side $f$ and $g$ given by the exact solution
are computed on a 
sequence of red-refined triangulations. The initial triangulation is 
depicted in Figure~\ref{f:inittriangcube}.
The $H^1$ errors and the $L^2$ errors of the velocity $u$ are depicted 
in the convergence history plot in Figure~\ref{f:cube3}. 
The $H^1$ errors show convergence rates slightly worse than $\mathcal{O}(h)$ for all three 
methods. As the convergence rate still increases under the considered 
refinements, it is suggested that the asymptotic regime is not reached at 
this point.
The convergence rates of the $L^2$ errors of all three methods 
are $\mathcal{O}(h^2)$ for all methods. 

\begin{figure}
 \begin{center}
   \includegraphics[width=0.95\textwidth]{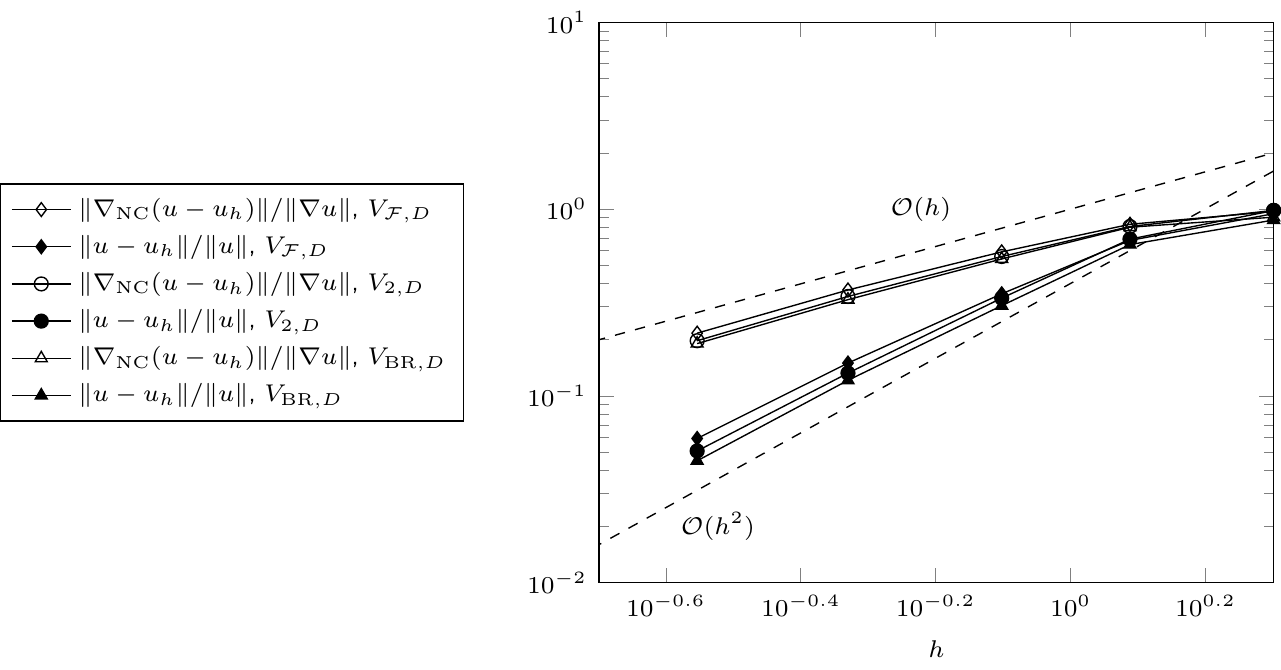}
 \end{center}
 \caption{\label{f:cube3}Convergence history plot for the example from 
Subsection~\ref{ss:numericsCube3}}
\end{figure}

\subsection{Singular solution on the 3D tensor product L-shaped domain}
\label{ss:numericsLG3}

\begin{figure}
 \begin{center}
   \includegraphics[width=0.95\textwidth]{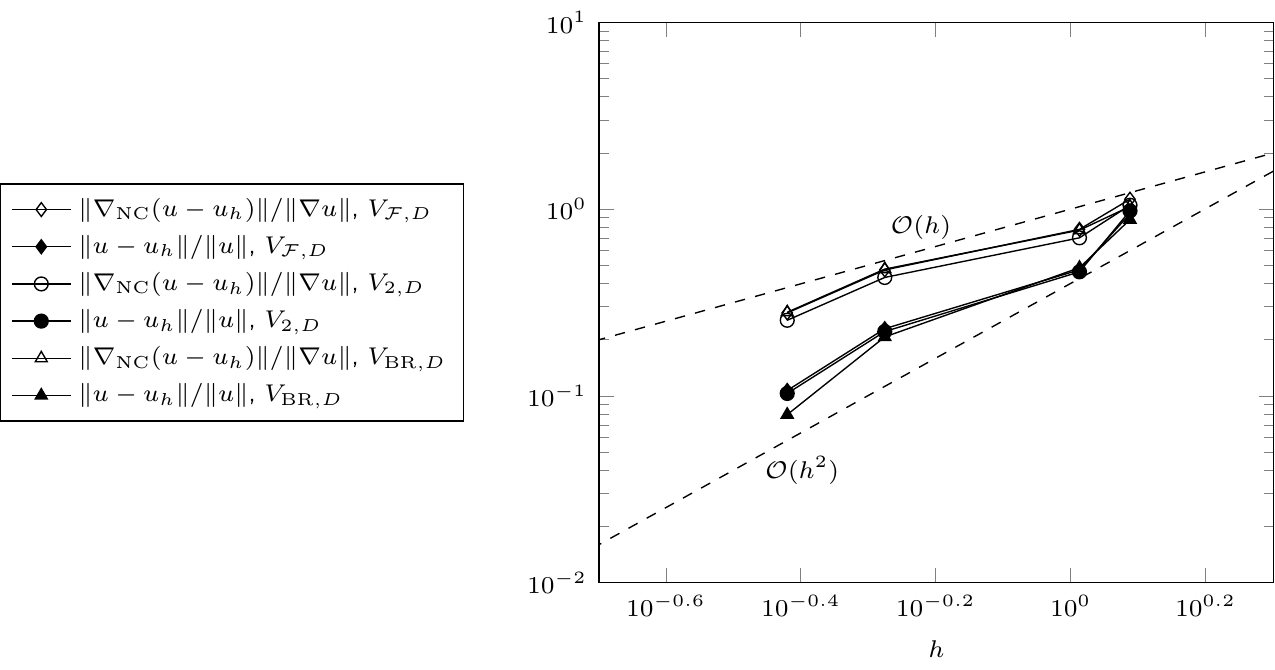}
 \end{center}
 \caption{\label{f:LG3}Convergence history plot for the example from 
Subsection~\ref{ss:numericsLG3}}
\end{figure}

This subsection considers the tensor product L-shaped domain 
\begin{align*}
  \Omega=\left((-1,1)^2\setminus([0,1]\times[-1,0]) \right)\times (-1,1)
\end{align*}
with $\Gamma_N:=\{1\}\times (0,1)\times (-1,1)$
and exact solution 
\begin{align*}
 u(x_1,x_2,x_3) 
   = \Curl\begin{pmatrix}
            (1-x_3^2)^2 u_\mathrm{Gr}(x_1,x_2)\\
            \cos(\pi x_3) u_\mathrm{Gr}(x_1,x_2)\\
            x_3 u_\mathrm{Gr}(x_1,x_2),
          \end{pmatrix}
  \qquad\text{and}\qquad 
  p=0,
\end{align*}
where $u_\mathrm{Gr}$ is the singular solution for the 2d plate problem
on the L-shaped domain from~\cite[p.~107]{Grisvard1992} and reads in 
polar coordinates
\begin{align*}
  u(r,\theta)=(r^2\cos^2\theta-1)^2 (r^2\sin^2\theta-1)^2\, r^{1+\alpha}\, g(\theta).
\end{align*}
Here, $\alpha:=0.544483736782464$ is a noncharacteristic root 
of $\sin^2(\alpha\omega)=\alpha^2\sin^2(\omega)$ for $\omega:=3\pi/2$ and 
\begin{align*}
  g(\theta) &= \left[\frac{\sin((\alpha-1)\omega)}{\alpha-1}-\frac{\sin((\alpha+1)\omega)}{\alpha+1}\right]
     (\cos((\alpha-1)\theta)-\cos((\alpha+1)\theta)\\
   &\qquad 
    - \left[\frac{\sin((\alpha-1)\theta)}{\alpha-1}-\frac{\sin((\alpha+1)\theta)}{\alpha+1}\right]
     (\cos((\alpha-1)\omega)-\cos((\alpha+1)\omega).
\end{align*}
The right-hand side data $f$ and $g$ are chosen according to the exact solution.
The initial triangulation is depicted in Figure~\ref{f:inittriangcube}.
The $H^1$ and $L^2$ errors are plotted in Figure~\ref{f:LG3} against the mesh-size. 
Although the exact solution is not in $H^2(\Omega)$, the convergence rate 
of the $H^1$ errors for all three methods seems to be $\mathcal{O}(h)$, 
at least in a pre-asymptotic regime. This is in agreement with numerical 
experiments in 2d and 3d for the plate problem, where the reduced convergence 
rate can only be seen in the regime of very fine meshes. The $L^2$ errors show convergence 
rates slightly smaller than $\mathcal{O}(h^2)$ for the three considered methods,
but the Bernardi-Raugel FEM seems to have a slightly better convergence rate
than the two nonconforming FEMs.


{\footnotesize
\bibliographystyle{alpha}
\bibliography{KS3D}
}

\end{document}